\documentclass[runningheads,envcountsame,letterpaper]{llncs}

\usepackage{tkz-euclide}
\tikzstyle{directed}=[postaction={decorate,decoration={markings,
    mark=at position .65 with {\arrow{stealth}}}}]
\tikzstyle{reverse directed}=[postaction={decorate,decoration={markings,
    mark=at position .5 with {\arrowreversed{stealth};}}}]
\usepackage{amssymb}
\usepackage{latexsym}
\usepackage{amsmath}
\usepackage{xspace}
\usepackage{mathrsfs}
\usepackage[all]{xy}
\usepackage{enumerate}
\usepackage{amsmath,amssymb,tabularx,hhline,rotating,enumerate,xspace}
\usepackage{psfrag}
\usepackage{ifthen}
\usepackage[utf8]{inputenc}
\usepackage{tabularx} 
\usepackage{graphicx}
\usepackage{adjustbox}
\usepackage{mathtools}
\usepackage{hyperref}
\usepackage[capitalize]{cleveref}

\newenvironment{cmv}{\noindent\color{violet}}{}

\makeatletter
\renewcommand*\env@matrix[1][*\c@MaxMatrixCols c]{%
  \hskip -\arraycolsep
  \let\@ifnextchar\new@ifnextchar
  \array{#1}}
\makeatother

\newenvironment{cm}{\noindent\color{blue}}{} 

\newcommand{\comment}[1]
   {\ifthenelse{\equal{\showcomments}{yes}}
	     {\footnotemark\marginpar{\sffamily{\tiny
		   \addtocounter{footnote}{-1}\footnotemark#1
}\normalfont}}{}}

\newcommand{\NP}{\ensuremath{\mathsf{NP}}\xspace} %

\newcommand{\LOGSPACE}{\ensuremath{\mathsf{LOGSPACE}}\xspace} %
\newcommand{\NL}{\ensuremath{\mathsf{NL}}\xspace} %
\newcommand{\ACC}{\ensuremath{\mathsf{ACC}^0}\xspace} %

\renewcommand{\P}{\ensuremath{\mathsf{P}}\xspace}

\newcommand{\showcomments}{yes}
\newcommand*{\DecProblem}[1]{{\textsc{#1}}\xspace}
\newcommand{\DPSE}[1]{\DecProblem{SphEq}(#1)}

\newcommand{\bel}[1]{\begin{equation}\label{#1}}
\newcommand{\ee}{\end{equation}}

\newcommand{\Z}{\mathbb{Z}}
\newcommand{\N}{\mathbb{N}}

\newcommand{\LBA}{\left\{\begin{array}}
\newcommand{\EAR}{\end{array}\right.}

\newcommand{\UT}[2]{\operatorname{UT}(#1,#2)}
\newcommand{\SL}[2]{\operatorname{SL}(#1,#2)}
\newcommand{\GL}[2]{\operatorname{GL}(#1,#2)}
\newcommand{\TL}[2]{\operatorname{T}(#1,#2)}
\newcommand{\ET}[2]{\operatorname{ET}(#1,#2)}

\def\CG{{\mathcal G}}

\newcommand{\gen}[1]{\left< \mathinner{#1} \right>}
\newcommand{\Gen}[2]{\left< \mathinner{#1} \mid \mathinner{#2}\right>}

\newcommand{\abssmall}[1]{\lvert\mathinner{#1}\rvert}
\newcommand{\interval}[2]{[ \mathinner{#1}..\mathinner{#2}] }

\newcommand{\compproblem}[3][]{%
	\par\vspace{0.125cm plus 0.1cm minus 0.05cm}\adjustbox{valign=t}{\begin{tabularx}{\linewidth-2\parindent}{@{}lX}%
			\if\relax\detokenize{#1}\relax%
			\else%
			\textnormal{\textsf{Constant:}}&#1\\%
			\fi%
			\textnormal{\textsf{Input:}}&#2\\%
			\textnormal{$\mathrlap{\textsf{Output:}}\hphantom{\textsf{Question:}}$\!\!}&#3\\%
	\end{tabularx}}\vspace{0.125cm plus 0.1cm minus 0.05cm}\par%
}
\newcommand{\ynproblem}[3][]{%
	\par\vspace{0.125cm plus 0.1cm minus 0.05cm}\adjustbox{valign=t}{\begin{tabularx}{\linewidth-2\parindent}{@{}lX}%
			\if\relax\detokenize{#1}\relax%
			\else%
			\textnormal{\textsf{Constant:}}&#1\\%
			\fi%
			\textnormal{\textsf{Input:}}&#2\\%
			\textnormal{\textsf{Question:}}&#3\\%
	\end{tabularx}}\vspace{0.125cm plus 0.1cm minus 0.05cm}\par%
}

\newcommand{\prblminput}[2][]{%
	\par\vspace{0.125cm plus 0.1cm minus 0.05cm}\adjustbox{valign=t}{\begin{tabularx}{\linewidth-\parindent}{@{}lX}%
			\if\relax\detokenize{#1}\relax%
			\else%
			\textnormal{\textsf{Constant:}}&#1\\%
			\fi%
			\textnormal{\textsf{Input:}}&#2%
		\end{tabularx}}
}

\newcommand{\abs}[1]{\left|{}#1\right|}

\newcommand{\Set}[2]{\left\{\, #1 \;\middle|\; #2 \,\right\}}
\newcommand{\oneset}[1]{\left\{\, #1 \,\right\}}

\newcommand{\cH}{\mathcal{H}}
\newcommand{\cC}{\mathcal{C}}

\newcommand{\cG}{\mathcal{G}}

\newcommand{\cP}{\mathcal{P}}
\newcommand{\prm}{\mathbb{P}}

\newcommand{\Oh}{\mathcal{O}}
\def\MN{{\mathbb{N}}}
\def\MZ{{\mathbb{Z}}}

\newcommand{\sgn}[1]{\mathrm{sgn}(#1)}

\DeclareMathOperator{\tr}{{tr}}
\DeclareMathOperator{\Aut}{{Aut}}

\title{Complexity of Spherical Equations in Finite Groups}
\author{Caroline Mattes\inst{1}, Alexander Ushakov\inst{2}, Armin Wei\ss\inst{1}}
\institute{
	Universität Stuttgart, Institut für Formale Methoden der Informatik, Germany\and
 Stevens Institute of Technology, Hoboken, USA
 }
 
\date{\today}

\begin{document}

\maketitle

\begin{abstract}
In this paper we investigate computational properties
of the Diophantine problem for spherical equations
in some classes of finite groups.
We classify the complexity of different variations of the problem, 
e.g., when $G$ is fixed and when $G$ is a part of the input. 

When the group $G$ is constant or given as multiplication table, we show that the problem always can be solved in polynomial time. On the other hand, for the permutation groups $S_n$ (with $n$ part of the input), the problem is \NP-complete. 
The situation for matrix groups is quite involved: while we exhibit sequences of 2-by-2 matrices where the problem is \NP complete, in the full group $\GL{2}{p}$ ($p$ prime and part of the input) it can be solved in polynomial time.
We also find a similar behaviour with subgroups of matrices of arbitrary dimension over a constant ring.
\end{abstract}

\section{Introduction}

The study of equations has a long history in all of mathematics. 
Some of the first explicit decidability results in group theory are due to Makanin
\cite{mak84}, who showed that equations over free groups are decidable.
Let $F(Z)$ denote the free group on countably many generators 
$Z = \{z_i\}_{i=1}^\infty$. For a group $G$, an 
\emph{equation over $G$ with variables in $Z$} is a formal equality
of the form $W = 1$, where 
\[
W = z_{i_1}g_1\cdots z_{i_k}g_k \in F(Z) \ast G, 
\mbox{ with  $z_{i_j}\in Z$ and $g_j\in G$}
\]
and $\ast$ denotes the free product.
We refer to $\{z_{i_1},\ldots,z_{i_k}\}$ as the set of \emph{variables} and to $\{g_1,\ldots,g_k\}$ as the set of 
\emph{constants} (or \emph{coefficients}) of $W$.
We occasionally write $W(z_1,\ldots,z_k)$ to indicate that the variables in $W$ are precisely $z_1,\ldots,z_k$.
A \emph{solution} to an equation $W(z_1,\ldots,z_k)=1$ over $G$ is an
assignment to the variables $z_1,\ldots,z_k$ that makes $W(z_1,\ldots,z_k)=1$ true.
The \emph{Diophantine problem} (also called \emph{equation satisfiability} or \emph{polynomial satisfiability} problem) in a group $G$ for a class of equations $C$ 
is an algorithmic question to decide if a given equation $W=1$
in $C$ has a solution, or not.

One class of equations over groups that has generated much interest 
is the class of \emph{quadratic} equations: equations where each 
variable $z$ appears exactly twice (as either $z$ or $z^{-1}$). 
The relationship between quadratic equations and compact surfaces sparked the initial interest in their study.
It was observed in \cite{Culler:1981} and \cite{Schupp:1980} 
in the early 80s that such equations have an affinity with the theory of compact surfaces (for instance, via their associated van Kampen diagrams). 
This geometric point of view has led to many interesting results,
particularly in the realm of quadratic equations over free groups: 
solution sets were studied  in \cite{Grigorchuk_Kurchanov:1992}, 
$\NP$-completeness was proved in 
\cite{Diekert-Robson:1999,Kharlampovich-Lysenok-Myasnikov-Touikan:2010}.
Systems of quadratic equations
played an important role in the study of the first order theory of free groups 
(Tarski problem, \cite{Kharlampovich_Myasnikov:1998(1)}).
These results stimulated more interest in computational properties
of quadratic equations in various classes of (infinite) groups
such as hyperbolic groups 
(solution sets were described in \cite{Grigorchuk-Lysionok:1992},
$\NP$-completeness was proved in \cite{Kharlampovich-Taam:2017}),
the first Grigorchuk group
(decidability was proved in \cite{Lysenok-Miasnikov-Ushakov:2016},
commutator width was computed in \cite{Bartholdi-Groth-Lysenok:2022}),
free metabelian groups
($\NP$-hardness was proved in \cite{Lysenok-Ushakov:2016},
membership to $\NP$ for orientable equations was 
proved in\cite{Lysenok-Ushakov:2021}),
metabelian Baumslag--Solitar groups
($\NP$-completeness was proved in \cite{Mandel-Ushakov:2023}), etc.

Let us introduce some terminology related to quadratic equations.
We say that equations $W = 1$ and $V=1$ are \emph{equivalent} 
if there is an automorphism $\phi\in \Aut(F(Z)\ast G)$ such that $\phi$ is the identity on $G$ and $\phi(W) = V$. It is a well known consequence of the classification of compact surfaces that any quadratic equation over $G$ is equivalent, via an automorphism $\phi$ computable in time $O(|W|^2)$, to an equation in exactly one of the following three standard forms (see, 
\cite{Comerford_Edmunds:1981,Grigorchuk_Kurchanov:1992}):
\begin{align}
\prod_{j=1}^k z_j^{-1} c_j z_j&=1, &&&(k\ge 1),\label{eq:spherical} \\
\prod_{i=1}^g [x_i,y_i]\prod_{j=1}^k z_j^{-1} c_j z_j&=1, &
\prod_{i=1}^g x_i^2    \prod_{j=1}^k z_j^{-1} c_j z_j&=1, &(g\geq 1, k\geq 0).
\label{eq:orientable}
\end{align}

The number $g$ is the \emph{genus} of the equation, and both $g$ and $k$ 
(the number of constants) are invariants. 
The standard forms are called, respectively, 
\emph{spherical} (\ref{eq:spherical}), 
\emph{orientable of genus $g$} and 
\emph{non-orientable of genus $g$}
(\ref{eq:orientable}) according to the compact surfaces they correspond to. 

In this paper we investigate spherical equations in finite groups. 
Let us remark at this point that spherical equations naturally generalize fundamental (Dehn) problems of group theory such as the word and the conjugacy problem if we allow that the constants are given as words over the generators.
The complexity of solving equations in finite groups has been first studied by Goldmann and Russell showing that the Diophantine problem in a fixed finite nilpotent group can be decided in polynomial time, while it is \NP-complete in every finite non-solvable group. 
For some recent results, see \cite{FoldvariH20,IdziakKKW22}.

\paragraph{Contribution.} Here we study the complexity of solving spherical equations in finite groups both in the case that the group is fixed as well as that it is part of the input. In the latter case, we consider various different input models as well as different classes of groups. In particular we show (here $\DPSE{\cG}$ denotes the Diophantine problem for the class $\cG$ of groups):
\begin{itemize}
    \item In a fixed group the satisfiable spherical equations form a regular language with commutative syntactic monoid~-- in particular, the Diophantine problem can be decided in linear time (\cref{thm: finitegroup}).
    \item If the Cayley table is part of the input, we can solve the Diophantine problem in nondeterministic logarithmic space and, thus, in \P (\cref{prop:CayNL}).
    \item For (symmetric) groups given as permutation groups the problem is \NP-complete (\cref{thm: DPforSn}).
    \item For two-by-two matrices we prove the following: 
     \begin{itemize}
 \item \label{ETNPcompl}  $\DPSE{(\ET{2}{n})_{n \in \N}}$ is $\NP$-complete (here, $\ET{2}{n}$ denotes the upper triangular matrices with diagonal entries in $\{\pm1\}$, \cref{lem: DninTwo}). 
 \item \label{TLNOcompl}  $\DPSE{(\TL{2}{p})_{p \in \prm}}$ is in $\P$   (here, $\TL{2}{p}$ are the upper triangular matrices and $\prm$ the set of primes, \cref{thm:TinP}). 
 \item  \label{GLNOcompl} $\DPSE{(\GL{2}{p})_{p \in \prm}}$ is in $\P$ (\cref{thm: GLinP}). 
 \end{itemize}

    \item $\DPSE{(\UT{4}{p})_{p \in \prm}}$ is in \P where $\UT{4}{p}$ denotes the  uni-triangular 4-by-4 matrices modulo $p$ (\cref{thm:UTfour}).
    \item $\DPSE{(H_n^{(p)})_{(n,p) \in \N \times \prm}}$ is in \P (here $H_n^{(p)}$ denotes the $n$-dimensional discrete Heisenberg group modulo $p$~-- note that the dimension and the prime field are both part of the input, \cref{thm:heisenberg}).
    \item For every constant $m \geq 5$, we exhibit sequences of subgroups $G_n'\leq H_n\leq G_n \leq \GL{n}{m}$ such that $\DPSE{(G_n')_{n \in \N}}$ and $\DPSE{(G_n)_{n \in \N}}$ are in \P but $\DPSE{(H_n)_{n \in \N}}$ is \NP-complete (\cref{prop: NPhardSubclass }). In particular, \NP-hardness does not transfer to super- nor sub-groups!
\end{itemize}

\section{Notations and problem description}\label{sec:NotProblm}
For some finite alphabet $\Sigma$, we denote the set of words over $\Sigma$ by $\Sigma^*$.
By $\prm$ we denote the set of all prime numbers. With $\interval{k}{\ell}$ for $k,\ell \in \Z$ we denote the interval of integers $\{k, \dots, \ell\}$.

\subsection{Notations from group theory}

 We assume the reader to be familiar with the basics of group theory. 
In this paper, we mostly work with  finite permutation or matrix groups, but in some cases we also use other forms of presentations, for example for the dihedral groups. Remember that we can find every finite group $G$ as a subgroup of some permutation or  matrix group (\cite [(7.4) and (9.13)]{SuzukiGTI81}).

A group can also be presented by its \emph{Cayley table} (or \emph{multiplication table}). The Cayley table of a finite group $G$ is a square matrix with rows and columns indexed by the elements of $G$. Row $g$ and column $h$ contains the product $k = gh \in G$ (i.e.\ the Cayley table is the full description of the group multiplication $G \times G \to G$). 

We will further use the following notations: Let $G$ be a group and $g_1, \ldots, g_\ell \in G$. Then, $\gen{g_1, \ldots, g_\ell}$ denotes the subgroup of $G$ generated by $g_1, \ldots, g_\ell$. 
For $g \in G$ we say that $g$ is conjugated to $h$ in $G$ if there exists $z \in G$ such that $z^{-1}gz=h$. In this case we write $g^z=h$. 

By $\Z_n$ we denote the ring $\Z/n\Z$ of residue classes modulo $n$. As a group we consider $\Z_n$ as the additive cyclic group of order $n$. We also denote the cyclic group of order two by $C_2$ when we write it multiplicatively as $\{1,-1\}$.

\newcommand{\supp}[1]{\mathrm{mov}(#1)}
\newcommand{\Ssupp}[1]{\mathrm{smov}(#1)}
\newcommand{\cyc}[1]{\mathrm{Cyc}(#1)}

\paragraph*{Permutation groups.}

The symmetric group on $n$ elements, denoted by $S_n$, is the group of order $n!$ of all bijective maps (aka. \emph{permutations}) on the set $\interval{1}{n}$. 
There are many ways to represent the elements of $S_n$. We will use the presentation  based on \emph{cycles}. A cycle is a tuple $(i_1, \ldots, i_k)$ with $i_j$ being pairwise disjoint numbers in $\interval{1}{n}$. As a function, such a cycle maps $i_j \mapsto i_{j+1}$ and $i_k \mapsto i_1$. We can write each element in $S_n$ as a product of cycles (i.e. the composition of the functions described by these cycles).

 Let $\sigma, \sigma_1, \ldots, \sigma_k  \in S_n$. Then we define 
\begin{enumerate}[(a)]
    \item $\supp{\sigma}= \Set{i \in \interval{1}{n}}{\sigma(i)\neq i} $
    \item $ \Ssupp{\sigma_1, \ldots, \sigma_k}=\sum_{i=1}^k \abs{\supp{\sigma_i}} $
\end{enumerate}
 We say that two cycles $\sigma_1,\sigma_2 \in S_n$ are disjoint, if $\supp{\sigma_1}\cap \supp{\sigma_2}=\emptyset$.

\begin{lemma}[\cite{AschbacherFGT}, (15.2)] \label{lem: DisjCycSn}
\begin{enumerate}[(a)]
    \item Disjoint cycles in $S_n$ commute. 
    \item \label{UDisDec}  Every element in $S_n$ can be written as the product of disjoint cycles in a unique way (up to the order of the cycles). This is called the disjoint cycle representation.
\end{enumerate}
\end{lemma}
We say that $\cyc{\sigma}=[j_1, \ldots, j_\ell]$ is the cycle structure of $\sigma$ if it is the product of $\ell$ disjoint cycles of lengths $j_1, \ldots, j_\ell$ and denote by $\abs{\cyc{\sigma}}=\ell$ the length of this cycle decomposition. Observe that the $j_i$ do not have to be pairwise disjoint. 

\begin{lemma}\label{lem: SumSuppSn} 
Let $\sigma_1,\sigma_2 \in S_n$.  Then, $\abs{\supp{\sigma_1\sigma_2}}\leq \Ssupp{\sigma_1, \sigma_2}$. Moreover,  $\supp{\sigma_1}\cap \supp{\sigma_2}\neq \emptyset$ if and only if $\abs{\supp{\sigma_1\sigma_2}} < \Ssupp{\sigma_1, \sigma_2}$. 
\end{lemma}

\begin{proof}
If $j \notin \supp{\sigma_1}\cup \supp{\sigma_2}$, then $\sigma_1(\sigma_2(j))=j$. Therefore, $ \supp{\sigma_1\sigma_2} \subseteq \supp{\sigma_1}\cup \supp{\sigma_2}$. If $\supp{\sigma_1}\cap \supp{\sigma_2}\neq \emptyset$, then $\abs{\supp{\sigma_1}\cup \supp{\sigma_2}}<\Ssupp{\sigma_1,\sigma_2}$. If $\supp{\sigma_1}\cap \supp{\sigma_2}= \emptyset$, then $j \in \supp{\sigma_1} $ implies $\sigma_2(j)=j$ and vice versa. Thus, $ \supp{\sigma_1\sigma_2}=\supp{\sigma_1}\cup \supp{\sigma_2}$ and the lemma follows.  
 \qed\end{proof}

\begin{lemma}[\cite{AschbacherFGT}, (15.3)]\label{lem: ConjClassSn}
Let $\sigma, \tau \in S_n$. Then $\sigma$ is conjugated to $\tau$ in $S_n$ if and only if $\cyc{\sigma}=\cyc{\tau}$ (up to a permutation of the $j_i$).  
\end{lemma}
We can write each element in $\sigma \in S_n$ as a product of cycles of length two, so-called \emph{transpositions} (\cite[15.4]{AschbacherFGT}). If $\sigma$ is a product of an even number of transpositions, then $\sigma$ is an \emph{even} permutation and we write $\sgn{\sigma}=1$. Otherwise, $\sigma$ is an \emph{odd} permutation and $\sgn{\sigma}=-1$. 

\begin{lemma}[\cite{AschbacherFGT}, (15.5)] \label{lem: OddaEvenPerm}
    A permutation is even (rsp. odd) if and only if it is the product of an even (rsp. odd) number of cycles of even length. 
\end{lemma}
In particular, each $\sigma \in S_n$ is either odd or even. 
The set of all even permutations forms a normal subgroup of order $n!/2$ of $S_n$ usually denoted by $A_n$.

\paragraph*{Matrix groups.}\label{sec: matrices}
 The \emph{general linear group} $\GL{n}{p}$ over 
the prime field $\MZ_p$ is the group of all invertible $n$-by-$n$ matrices. Recall that a matrix $A$ is invertible if and only if its \emph{determinant} $\det(A)\neq 0$.
By $\SL{n}{p}$ we denote the \emph{special linear group} of all matrices $A \in \GL{2}{p}$ with $\det(A)=1$. Further, let $\TL{n}{p}$ be the subgroup of $\GL{n}{p}$ of all invertible upper triangular matrices. By $\ET{n}{p}$ we denote the subgroup of $\TL{n}{p}$ of all matrices having only $\pm 1$ on the diagonal (note that this is not a standard definition). Finally, let $\UT{n}{p}$ be the group of all matrices $A \in \TL{n}{p}$  having  only $1$ on the diagonal.
We denote the identity matrix by $I$ or $I_n$.
\newcommand{\partition}{\DecProblem{Partition}}
\newcommand{\threepart}{\DecProblem{3-Partition}}
\newcommand{\esc}{\DecProblem{3-ExactSetCover}}
\subsection{Complexity and \NP-complete problems}\label{sec:NP}
We assume the reader to be familiar with the complexity classes \NL (nondeterministic logarithmic space), \P and \NP and the concept of \NP-completeness (we use \NP-completeness with respect to \LOGSPACE reductions). 

We say a function or problem is in random polynomial time, if there is a randomized polynomial time algorithm computing the correct answer with a high probability. Since in our case  we always can check in (deterministic) polynomial time whether the answer is correct, we can modify the algorithm to obtain a polynomial time Las Vegas algorithm (meaning an algorithm which is always correct and the expected running time is polynomial).

It is well known that the following problems are \NP-complete (see \cite{GarayJ79}): 

\smallskip
 \noindent\threepart: 
    \ynproblem{A list $S=(a_1, \ldots, a_{3k})$ of unary-encoded positive integers  such that $\frac{L}{4}<a_i<\frac{L}{2}$ with $L=\frac{1}{k}\sum_{i=1}^{3k}a_i$.}{Is there a partition of $\interval{1}{3k}$ into $k$ (disjoint) sets $A_i=\{i_1, i_2, i_3 \}$ with  $a_{i_1}+a_{i_2}+a_{i_3}=L$ for all $i \in \interval{1}{k}$ ? }
    \smallskip
\noindent\partition: 
    \ynproblem{A list $S=(a_1, \ldots, a_{k})$ of binary-encoded positive integers.}{Is there a subset $A \subseteq \interval{1}{k}$  such that $\sum_{i\in A} a_i = \sum_{i \not\in A} a_i$? }
Recall that we need to give the numbers $a_i$ in binary as otherwise \partition can be solved in polynomial time using a dynamic programming approach.
    \smallskip
    
\noindent\esc: 
    \ynproblem{A set $A$ and subsets $A_1, \ldots, A_\ell \subseteq A$ with $\abs{A_i}\leq 3$.}{Is there a set $I \subseteq \interval{1}{\ell}$ such that $\bigcup_{i \in I} A_i = A$ and $A_i \cap A_j =\emptyset$ for all $i \neq j $ in $I$? }
    The problem is still \NP-hard if each element occurs in at most $3$ sets $A_i$.

\subsection{Spherical Equations}
Let $G$ be a group. A \emph{spherical equation} over $G$ is an equation 
\begin{align}\label{DefSE}
\prod_{i=1}^k z_i^{-1}c_iz_i=1
\end{align}
with $c_i \in G$ and $z_i$ being variables for $i \in \interval{1}{k}$. 
As the length of such an equation we define the number of  $c_i \neq 1$ in $G$. Note that these are quadratic equations because each $z_i$ appears exactly twice (as $z_i$ and $z_i^{-1}$).   

Let $G$ be any finite group. Then we write:
\[\DPSE{G} = \{(c_1, \ldots, c_k) \in G^k \mid k\in \N, \exists z_i \in G : \prod_{i=1}^k z_i^{-1}c_iz_i=1\}\]
Note that this is a formal language over the alphabet $G$ (we write $(c_1, \ldots, c_k)$ instead of $c_1 \cdots c_k$ as usual for formal languages to differentiate it from the product $c_1 \cdots c_k$ in $G$).
Using this notation, we write $(c_1, \ldots, c_k) \in \DPSE{G}$ if there exists a solution.
For $c \in G$ we also consider equations of the form
\begin{align}\label{AlternativeSE}
    \prod_{i=1}^{k} z_i^{-1} c_i z_i=c. 
\end{align}
This equation is equivalent to a spherical equation as in (\ref{DefSE}):
Conjugating both sides of (\ref{AlternativeSE}) by (the variable)  $z$, multiplying them by $z^{-1}cz$ and inserting  $zz^{-1}$ we obtain an equation of the form (\ref{DefSE}). These operations preserve the existence of a solution.  

We can identify $\DPSE{G}$ with the computation problem 
\ynproblem{ $c_1, \ldots, c_k \in G$.}{Is $(c_1, \ldots, c_k) \in \DPSE{G}$? }

In the following we want to consider this problem also with the group being part of the input. There are several ways how to represent a finite group as part of the input:
\newcommand{\DPfix}[2]{\textsc{SphEq}_{#1}(#2)}
\newcommand{\DPCT}{\textsc{SphEq}_{\mathrm{CayT}}}
\newcommand{\DPSub}[1]{\textsc{SphEq}(\textsc{Sgr}(#1))}

\vspace{-4mm}
\subsubsection{Input models for finite groups.} 
For a class/sequence of groups $\cG=(G_n)_{n\in \N}$ we define the \emph{satisfiability problem} (or \emph{Diophantine problem}) for spherical equations for $\cG$, denoted by  $\DPSE{\cG}$, as follows:
\ynproblem{ $n \in \N$, a description of $G_n$ and elements $c_1, \ldots, c_k \in G_n$.}{Is $(c_1, \ldots, c_k) \in \DPSE{G_n}$? }
We will use the following representations of the input, which are frequently used in literature and computational algebra software \cite{GAP4}: 
\vspace{-3mm}
\paragraph*{Constant.} 
In this model, $G$ is not part of the input but fixed. Thus, $\DPSE{G}$ is defined as above.
\vspace{-3mm}
\paragraph*{Cayley table.}
The problem $\DPCT$ has the following input: 
\prblminput{The Cayley table of a group $G$ and elements $c_1, \ldots, c_k$ in $G$.}
\vspace{-3mm}
\paragraph*{Permutation groups.} The problem $\DPSE{(S_n)_{n\in \N}}$ (resp. $\DPSE{(A_n)_{n\in \N}}$) has the following input:  
\prblminput{$n \in \N$ in unary and elements $c_1, \ldots, c_k \in S_n$ (resp. $A_n$) given as permutations. }
\vspace{-3mm}
\paragraph*{Matrix groups.} For matrix groups we have two parameters, the dimension $n$ and $q$ for the ring $\Z_q$. Both can be either fixed or variable. Usually, the dimension $n$ is given in unary and $q$ as well as the entries of the matrices are given in binary. We describe the input for $\DPSE{(\GL{n}{q})_{q \in \N}}$, but it is defined similarly for other matrix groups or if $q$ is fixed. 
\prblminput{$q \in \N$ in binary and matrices $c_1, \ldots, c_k$ in $\GL{n}{q}$.}
\vspace{-5mm}
\paragraph*{Subgroups.} Let $\cG$ be a sequence of groups (defined by matrices/permutations etc). By $\DPSub{\cG}$ we denote the following problem: 
\ynproblem{a description of a group $G \in \cG$, elements $g_1, \ldots, g_\ell$ in $G$ and $c_1, \ldots, c_k \in \gen{g_1, \ldots, g_\ell}$. }{Is $(c_1, \ldots, c_k) \in \DPSE{\gen{g_1, \ldots, g_\ell}} $ ?}
Note that this is a promise problem: for the question to be well-defined we need the promise that $c_1, \ldots, c_k \in \gen{g_1, \ldots, g_\ell}$. 
\vspace{-3mm}
\paragraph*{Custom input forms.} For some sequences of groups like dihedral we use other forms of input representations. We give further details in the corresponding sections.

\begin{remark} The following input model  also might be of interest (but is not considered in this paper): 
\prblminput{A finite presentation $\Gen{g_1, \ldots, g_n}{r_1, \ldots, r_m}$ of $G$ and  $c_1, \ldots, c_k \in G$ given as words over the $g_i$. }
\end{remark}
\vspace{-3mm}
\subsubsection{First observations.} 
Observe that if the $c_i$ are given as words over the generators, then $\DPSE{G}$ of length $1$ is the word problem and $\DPSE{G}$ of length $2$ is the conjugacy problem in $G$. So we can reduce the word problem and the conjugacy problem in $G$ to $\DPSE{G}$ in this case. 
\smallskip

\begin{lemma}\label{lem:reorder}
Let $\pi:\oneset{1,\dots, k} \to \oneset{1,\dots, k}$ be a permutation. Then  $\prod_{i=1}^k z_i^{-1}c_iz_i=1$ has a solution if and only if $\prod_{i=1}^k \tilde z_{i}^{-1}c_{\pi(i)}\tilde z_{i}=1$ has a solution. In other words, $(c_1, \ldots, c_k) \in \DPSE{G}$ if and only if $(c_{\pi(1)}, \ldots, c_{\pi(k)}) \in \DPSE{G}$.
\end{lemma}

\begin{proof}
We have 
\begin{align*}
  x^{-1} c x \; y^{-1} d y =  x^{-1} c x \;  y^{-1} d y \;  x^{-1} c^{-1} x \; x^{-1} c x = \tilde y^{-1} d \tilde y \; x^{-1} c x
\end{align*}
for $\tilde y = y  x^{-1} c^{-1} x $. Thus, we can exchange $c_i$ and $c_{i+1}$. By induction we can apply any permutation to the $c_i$.
\qed\end{proof}

\begin{proposition}\label{prop:inNP}
Let $\cG$ denote any class/sequence of groups given as matrix or permutation groups as above. Then $\DPSE{\CG}$ and $\DPSub{\cG}$ are in \NP. 
\end{proposition}
Note that \cref{prop:inNP} actually also holds for arbitrary equations in finite groups.

\begin{proof} 
For $\DPSE{\CG}$, on input of $n \in \N$, $c_1, \ldots, c_k \in G_n$ we guess a solution $z_1, \ldots, z_k \in G_n$. 
Note that each $z_i$ is a matrix of polynomial dimension and with polynomially many bits per entry or an element of $S_n$ represented by  at most $n$ disjoint cycles in $S_n$. 
So each $z_i$ has polynomial size in the input and we can calculate $\prod_{i=1}^kc_i^{z_i}$ in polynomial time. 

Concerning $\DPSub{\cG}$, according to the reachability lemma (\cite[Lemma 3.1]{Babai84}) each element in $\gen{g_1, \ldots, g_\ell}$ has a representation as a straight-line program (for a definition, see \cite{Babai84}) of length at most $(1+\log(m))^2$ if $m=\abs{\gen{g_1, \ldots, g_\ell}}$. Because $\log(\abs{\GL{n}{p}})$ and $\log(\abs{S_n})$ are polynomial in $n$ or $p$ resp., for each $z_i$ we can guess a representation in the $g_i$ of length polynomial in the input. This shows the proposition. 
\qed\end{proof}

Notice that for each group $G$, $\DPSE{G}$ is non-empty. Indeed,  for  $k\geq 2$ and $c_1, \ldots, c_{k-1} \in G$ set $c_k=(\prod_{j=1}^{k-1}c_j)^{-1}$. On the other hand, $(c) \notin \DPSE{G}$ for all $1\neq c \in G$. For abelian groups the problem is easy: If $A$ is abelian, then $(c_1, \ldots, c_k) \in \DPSE{A}$ if and only if $\prod_{i=1}^k c_i=1$. 

\section{Fixed finite groups and Cayley tables}\label{sec:constantCayley}

\subsection{The world if $G$ is fixed}

In this section we consider  $\DPSE{G}$ if the group $G$ is not part of the input, but fixed. So we are in the \emph{constant} input model. We will show that $\DPSE{G}$ is in \ACC. For a definition we refer to \cite[Def.4.34]{Vollmer99}. Readers not familiar with circuit complexity might read it simply as \LOGSPACE.
Before we consider the complexity let us mention some interesting property of spherical equations in finite simple groups:

\begin{proposition}\label{lem: SimplGroup}
If $G$ is a finite non-abelian simple group, then every spherical equation of length at least $\abs{G}^3-\abs{G}+1$  has a solution (for the length we only count the number of non-trivial $c_i$).
\end{proposition}

For the proof, we need the following lemma:
\begin{lemma}\label{lem:prodofconjugats}
Let $G$ be a finite non-abelian simple group. Then there exists some $K \in \MN$ with $K \leq \abs G^2$ such that for every $h\in G$ and $g\ne 1$
there are $x_1, \dots, x_K \in G$ with $h = g^{x_1} \cdots g^{x_K}$.
\end{lemma}

\begin{proof}
First notice that it suffices to find for every $g \in G$ some $K_g \in \MN$ such that every $h \in G$ can be written as a product of conjugates of $g$ of length $K_g$. Then $K$ can be simply chosen as the maximum over the different $K_g$.

Let $m \in \N$ and $S(m) = \Set{g^{x_1} \cdots g^{x_m}}{x_1, \dots, x_m \in G}$. We have to show that $S(K_g) = G$ for some $K_g \leq \abs{G}^2$. Let $r$ denote the order of $g$. Then for all $k$ we have $S(k) \subseteq S(k+r)$. 

Hence, with increasing $k$ the sets $S(kr)$  eventually stabilize~-- i.e. there is some $k\leq \abs{G}$ with $S(kr) = S((k+i)r)$ for all $i \geq 0$. Clearly $1 \in S(kr)$. Moreover, $S(kr) \cdot S(kr) = S(2kr) = S(kr)$. Therefore, $S(kr)$ is a subgroup of $G$. Moreover, it is normal by its very definition.  So, $S(kr)=1$ or $S(kr)=G$. 
Thus, it remains to find some non-trivial element in $S(kr)$.

In order to do so, let $y \in G$ be some element that does not commute with $g$ (which exists because $G$ has trivial center).
Now define $x_1= \cdots =x_{kr-1}=1$ and $x_{kr}=y$. Then $g^{x_1} \cdots g^{x_{kr}}  = g^{kr-1} g^y = g^{-1} g^y \neq 1$. Hence, $S(kr) = G$.
Finally, note that $kr \leq \abs{G}^2$.
\qed\end{proof}

\begin{proof}[of \cref{lem: SimplGroup}]
Let $K$ be as in \cref{lem:prodofconjugats}. 
If we have a  spherical equation with  $(K-1) \cdot |G|+1$  non-trivial $c_i$, then at least one $c_i$  appears at least $K$ times.
Therefore, by \cref{lem:reorder} we can reorder the $c_i$ such that $c_1 = c_{2} = \cdots = c_{K}$. Now the lemma follows from \cref{lem:prodofconjugats}.
\qed\end{proof}

Next, let us turn to the complexity of solving spherical equations in a fixed finite group.
\begin{theorem} \label{thm: finitegroup}
Let $G$ be a fixed finite group. Then the set of satisfiable spherical equations in $G$ forms a regular language with commutative syntactic monoid. 
\end{theorem}
Note that here we view $G$ as a finite alphabet and identify the input $(c_1, \ldots, c_k)$ to a spherical equation with the word $c_1\cdots c_k \in G^*$ over the group elements. Be aware that here we do not allow the $c_i$ to be represented as word in $G^*$ but each of them needs to be given as a single group element.
Recall that a language is regular if and only if it has a finite syntactic monoid (or is recognized by a finite monoid). For background information and a definition of syntactic monoids, see e.g.\ \cite[Ch. 7.1]{edam16}.

\begin{proof}
Let  $\phi: G^* \to \cP(G)$ (the power set of $G$, which is a monoid) denote the  monoid homomorphism defined by $g \mapsto \Set{x^{-1}gx}{x \in G}$. 
Then by \cref{lem:reorder} the image of $\phi$ is commutative.
Moreover, $\prod_{i=1}^k c_i^{z_i}=1$ is satisfiable if and only if $1 \in \phi(c_1 \cdots c_k)$ as outlined above. Therefore, the set of satisfiable spherical equations is recognized by the homomorphism $\phi$. 
\qed\end{proof}

\begin{corollary} \label{cor: DPforFixedG}
$\DPSE{G}$  is in uniform $\ACC$ for every fixed finite group $G$. Moreover, it can be solved in linear time for every fixed finite group $G$.
\end{corollary}

\begin{proof}
By \cref{thm: finitegroup} the set of satisfiable spherical equations is recognized by a homomorphism into a commutative monoid. By \cite{McKenziePT91} all languages recognized by commutative (or, more generally, solvable) monoids are in $\ACC$. Moreover, as spherical equations are a regular language (they are recognized by a finite monoid), satisfiability can be decided in linear time.
\qed\end{proof}

\subsection{Multiplication tables and direct products for inputs}

In this section we prove that $\DPCT \in \NL$. Moreover, we will show that there are classes/sequences of groups $(G_n)_{n\in \N}$ such that $\abs{G_n}$ is exponential in $n$ but still, $\DPSE{(G_n)_{n\in \N}}$ is in \ACC.

\begin{proposition}\label{prop:CayNL}
 $\DPCT$ is in \NL. 
\end{proposition}

\begin{proof}
Let $G$ be a finite group and assume that the elements in $G$ are binary numbers $1, \ldots, \abs{G}$. We obtain an \NL-Algorithm for $\DPCT$ as follows: 

As input, we have the Cayley table of $G$ and elements $c_1, \ldots, c_k$.  Starting with $i=1$, subsequently guess $z_i$ and look for $c_i^{z_i}$ and $\prod_{j=1}^i c_j^{z_j}$ in the Cayley table. We only remember the last product (in each step, we only remember three group elements). Set the counter to $i$. After $k$ steps,  check if $\prod_{j=1}^k c_j^{z_j}=1$. 
\qed\end{proof}

\begin{proposition} \label{prop: finiteGroups}
    The problem $\DPCT$ can be solved in time $k\cdot \abs{G}^2$ on a random access machine.  
\end{proposition}
\begin{proof}
As above, we can assume that the elements in $G$ are the numbers $1, \ldots, \abs{G}$. Proceeding by induction on $k$,
we construct sets
\begin{align*}
V_j=\Set{\prod_{i=1}^j z_i^{-1}c_iz_i}{z_i\in G}.
\end{align*}
Such a set $V_j$ can be represented as a Boolean array of length $\abs{G}$. Each entry indicates whether the corresponding group element is contained in $V_j$.
 To construct $V_{j+1}$ from $V_j$ we need to multiply each element in $V_j$ with each element in $\{z^{-1}c_{j+1}z~|~z \in G\}$ and set the corresponding bit in $V_{j+1}$. Since $\abs{V_j}\leq \abs{G}$, we need at most $\Oh(\abs{G}^2)$ look-ups in the Cayley table.

Hence, we can construct the sets $V_1, \ldots, V_k$ in time $\Oh(k \cdot \abs{G}^2)$ and then check if $1 \in V_k$. \qed\end{proof}

    Note that, if $G,H$ are finite groups and $g_i \in G$, $h_i \in H$ for $i \in \interval{1}{k}$, then $((g_1, h_1), \ldots, (g_k, h_k)) \in \DPSE{G \times H}$ if and only if $(g_1, \ldots, g_k) \in \DPSE{G}$  and $(h_1, \ldots, h_k) \in \DPSE{H}$. Hence, we get the following corollary of \cref{cor: DPforFixedG} (when representing elements of $G^{n}$ as tuples of elements of $G$):

\begin{corollary}\label{lem: DirectProd}
  Let $G$ be a fixed finite group, and let $\CG=(G^{n})_{n \in \N}$, $n$ given in unary with $G^{n}=\underbrace{G \times \cdots \times G}_{n \text{ times }}$. Then,  $\DPSE{\CG} \in \ACC$. 
\end{corollary}

\begin{remark}\label{rem:}
    By \cref{prop: finiteGroups}, in any input model $\DPSE{(G_n)_{n\in \N}}$ is in \P if $\abs{G_n}$ is polynomial in the size of the description of $G_n$ given that the input model allows for a polynomial-time multiplication of group elements and the group elements are reasonably encoded (which applies to any reasonable input model). The reason is that first,  on input of the group $G_n$ and a spherical equation, we can compute the multiplication table of $G_n$ and then apply \cref{prop:CayNL}.
 On the other hand, by \cref{lem: DirectProd} there are classes/sequences of groups such that $\abs{G_n}$ is exponential in $n$ but $\DPSE{(G_n)_{n \in \N}}$ is in \ACC. 
\end{remark}

\section{The groups $S_n$ and $A_n$}

In this section we consider the \emph{permutation group} input model and we will show that $\DPSE{(S_n)_{n \in \N}}$ and $\DPSE{(A_n)_{n \in \N}}$ are $\NP$-complete. Observe that $\abs{S_n}$ and $\abs{A_n}$ are not polynomial in $n$. By \cref{prop:inNP} we already know that
    $\DPSE{(S_n)_{n \in \N}}$ and $\DPSE{(A_n)_{n \in \N}}$ are in \NP.

\begin{theorem}\label{thm: DPforSn}
 $\DPSE{(S_n)_{n \in \N}}$ is $\NP$-complete.
\end{theorem}

\newcommand{\sumA}[2]{\gamma^{#2}_{#1}}

\begin{proof}
To show $\NP$-hardness, we describe a reduction from \threepart (see \cref{sec:NP}):
Let $A=(a_1,\ldots,a_{3k})$ be an instance of \threepart with $L=\tfrac{1}{k}\sum_{i=1}^{3k} a_i$ and  $\frac{L}{4} < a_i < \frac{L}{2}$. We define cycles
$c_{1},\ldots,c_{3k} \in S_n$ for $n = k(L+1)$ and a product of disjoint $L+1$-cycles $c^\ast$ such that 
\begin{align}\label{defCandCstar}
    c_\ell &= (1,\ldots, a_{\ell}+1),\\
  \nonumber  c^\ast &= 
\prod_{i=1}^{k}
\bigl((i-1)(L+1)+1,\ldots,i(L+1)\bigr).
\end{align}
We claim that the spherical equation (see also (\ref{AlternativeSE}))
\begin{align} \label{EqRedSn}
    \prod_{i=1}^{3k} z_i^{-1} c_i z_i = c^\ast
\end{align}
has a solution in $S_n$ if and only if $A$ is a positive instance of \threepart. 

\newcommand{\ztv}{v}
Assume that we have a partition of  $\interval{1}{3k}$ into (disjoint) sets $A_i=\{i_1, i_2, i_3 \}$ with  $a_{i_1}+a_{i_2}+a_{i_3}=L$ (by a slight abuse of notation we use indices $i_1,i_2,i_3$ for elements of the set $A_i$). 
By \cref{lem: ConjClassSn} there exist $x_i, y_i \in S_n$ such that 
\begin{align*}
    x^{-1}_ic_{i_2}x_i&=(a_{i_1}+1,\ldots, a_{i_1}+a_{i_2}+1 ),\\
    y^{-1}_ic_{i_3}y_i&=(a_{i_1}+a_{i_2}+1,\ldots, a_{i_1}+a_{i_2}+a_{i_3}+1 ).
\end{align*}
Then,  $(1,  \ldots, L+1) = c_{i_1}\, x^{-1}_ic_{i_2}x_i\, y^{-1}_ic_{i_3}y_i  $. Thus there exists $\ztv_i$ such that
\begin{align*}
((i-1)\cdot (L+1)+1, \ldots,  i \cdot (L+1))
&=\ztv^{-1}_ic_{i_1}\ztv_i(\ztv_ix_i)^{-1}c_{i_2}(\ztv_ix_i)(\ztv_iy_i)^{-1}c_{i_3}(\ztv_iy_i).
\end{align*}
We obtain that
\begin{align*}
    \prod_{i=1}^k\ztv^{-1}_ic_{i_1}\ztv_i\; (\ztv_ix_i)^{-1}c_{i_2}(\ztv_ix_i) \; (\ztv_iy_i)^{-1}c_{i_3}(\ztv_iy_i)=c^\ast
\end{align*}
and hence, by \cref{lem:reorder} there exists a solution for (\ref{EqRedSn}). 
 \\

Now assume that (\ref{EqRedSn}) has a solution. Because $c^\ast$ is a product of $k$ disjoint cycles and there are $3k$ cycles $c_j^{z_j}$, there exist $i \neq j$ with  $\abssmall{\supp{c_i^{z_i}}}\cap\abssmall{\supp{c_j^{z_j}}}\neq \emptyset$.
Starting with $  \prod_{i=1}^{3k} c_i^{z_i}$,   we subsequently multiply neighbouring non-disjoint cycles and, if necessary, permute disjoint cycles. By \cref{lem: DisjCycSn} (\ref{UDisDec}), we eventually obtain the disjoint cycle decomposition of $c^\ast$.
Because $\abs{\cyc{c^\ast}}=k$ we need to multiply at least $2k$ non-disjoint cycles $\sigma_1, \sigma_2$ with $\abs{\cyc{\sigma_1\sigma_2}}=1$.
So, by \cref{lem: SumSuppSn}, $\Ssupp{\,\cdot\, }$ decreases by at least $2k$. Hence,
\begin{align*}
    k(L+1)=kL+3k-2k=\Ssupp{c^{z_1}_{1}, \ldots, c^{z_{3k}}_{3k}}-2k \geq \abs{\supp{c^\ast}}=k(L+1).
\end{align*}
Thus, no further multiplications of non-disjoint cycles are possible.  
This implies that for each $i \in \interval{1}{k}$ there exists $A_i=\{i_1, \ldots,i_{h_i}\}$ such that 
\begin{align*}
\prod_{\mu=1}^{h_i}z^{-1}_{i_\mu}c_{i_\mu}z_{i_\mu}=((i-1) \cdot (L+1) +1, \ldots, i(L+1)).
\end{align*} 

Thus, $\sum_{\mu=1}^{h_i}a_{i_\mu}\geq L$ and  $a_i<\frac{L}{2}$ imply $h_i \geq 3$ for $ i \in \interval{1}{k}$.
The index set has size $3k$, so $h_i=3$ for each $i$.
Because $k \cdot L = \sum_{i=1}^{3k} a_i $, we obtain $a_{i_1}+a_{i_2}+a_{i_3}=L$ for all $i$.
Thus, $A$ is a positive instance of \threepart. 
 \qed\end{proof}

\paragraph*{The alternating group $A_n$}

\begin{lemma}\label{lem: spltAn}
Let $c$  be a cycle of length $k$ in $A_n$ with $k \leq n-2$. Then, the conjugacy class of $c$ does not split in $A_n$ (i.e.\ $\{c^x \mid x\in A_n\} = \{c^x\mid x \in S_n\}$). 
\end{lemma}

We even have the following (see \cite[Exercises for Ch. 5]{AschbacherFGT}): For $c \in A_n$,  $\{c^x \mid x\in A_n\} \neq \{c^x\mid x \in S_n\}$  if and only if $\cyc{c}$ only consists of distinct odd numbers (be aware that here we have to also count the cycles of length one for points that are not moved). 
\begin{proof}
Let $\alpha, \beta \in S_n$ such that $\sgn{\alpha}=\sgn{\beta}=-1$. Because $k \leq n-2$ there exists $\tau \in S_n$ with $\supp{c}\cap \supp{\tau}=\emptyset$. This implies that $\tau c= c\tau$ and 
\begin{align*}
d=\beta^{-1} c\beta=\beta^{-1} c \tau^{-1} \tau\beta=
(\tau\beta )^{-1}c( \tau\beta)
\end{align*}
with $\sgn{\tau\beta}=1$.
So, if $c$ is conjugated to $d$ by some $\beta \in S_n \setminus A_n$ they are already conjugated in $A_n$. Thus the conjugacy class of $c$ does not split in $A_n$. \qed\end{proof}

\begin{corollary}\label{cor: SEAn}
     $\DPSE{(A_n)_{n \in \mathbb{N}}}$ is \NP-complete.
\end{corollary}

\begin{proof}
Notice that \threepart is still \NP-hard  if  we only consider instances $(a_1, \ldots, a_{3k})$ with all $a_i$ being even (just replace each $a_i$ by $2a_i$). Then, by \cref{lem: OddaEvenPerm}, $c_{i}$ and $c^*$   as  defined in  (\ref{defCandCstar}) are in $A_n$. Choosing $n = k(L+1) + 2$, the conjugacy class of each $c_{i}$ does not split in $A_n$ (\cref{lem: spltAn}). So we can use the same reduction as for $S_n$. 
\qed\end{proof}

Thus, there is a sequence of simple groups $\cG$ such that $\DPSE{\cG}$ is \NP-hard though \cref{lem: SimplGroup} suggests that it might be easy for simple groups. Notice that $|A_n|^3-|A_n|+1$ is not polynomial in $n$.

\begin{remark}
Observe that, if we consider the \emph{subgroup} input model,  \cref{thm: DPforSn} and \cref{cor: SEAn} imply that  $\DPSub{(S_n)_{n\in \N}}$ and $\DPSub{(A_n)_{n\in \N}}$ are \NP-complete.
\end{remark}

\section{Spherical equations in dihedral groups} \label{sec: Dihrgrps}
\newcommand{\Del}[1]{\Delta^{#1}}
The dihedral group $D_n$ of order $2n$ is given by the group presentation
\[D_n=\Gen{r,s}{r^n=s^2=1, srs=r^{-1}}.\]
We can embed $D_n$ into the group  $\GL{2}{n}$: 
Let $R=\small\begin{pmatrix}
1 & 1 \\
0 &1
\end{pmatrix}\normalsize$ and $ S=\small\begin{pmatrix}
1 &0 \\
0 &-1
\end{pmatrix}\normalsize$. Then, $D_n \cong \gen{R,S}$. Because $\gen{R}\cong \Z_n$ and $\gen{S}\cong C_2$ we get a presentation of  $D_n$ as semi-direct product: $D_n \cong \Z_n \rtimes C_2$ with $C_2$ acting on $\Z_n$ by $(k,\delta) \mapsto \delta k$.
So,  $D_n=\Set{(k,\delta)}{k\in\MZ_n,\ \delta=\pm1}$
 with multiplication defined by
\begin{align}\label{multDn}
    (k_1,\delta_1)(k_2,\delta_2)=
(k_1+\delta_1 k_2,\delta_1\cdot\delta_2).
\end{align}

\begin{lemma} 
Let $(h_\ell,\gamma_\ell), (k_\ell,\delta_\ell)\in D_n$ for $i \in \interval{1}{m}$. Then,
\begin{align}\label{FormSE}
\prod_{\ell=1}^{m}(h_\ell,\gamma_\ell)(k_\ell, \delta_\ell)(h_\ell,\gamma_\ell)^{-1}=\left(\sum_{\ell=1}^{m} \Delta^{\ell-1}\gamma_\ell k_\ell+\sum_{\ell=1}^{m}\Delta^{\ell-1} (1-\delta_\ell)h_\ell, \Delta^{m}\right)
\end{align}
\end{lemma}

\begin{proof}
Note that $(k,\delta)^{-1}=(-\delta k ,\delta)$. So we obtain 
\begin{align}\label{FormConj}
(h_1,\gamma_1)(k_1,\delta_1)(-\gamma_1 h_1 ,\gamma_1)
&=(h_1+\gamma_1 k_1-\gamma_1^2\delta_1 h_1,\gamma_1^2 \delta_1)\\\nonumber
&=((1- \delta_1)h_1+\gamma_1 k_1, \delta_1)
\end{align}

We show (\ref{FormSE}) by induction. For $m=1$ this is shown by  (\ref{FormConj}). By the induction hypothesis and (\ref{FormConj}) we obtain that
  \begin{align*}
 &\prod_{\ell=1}^{m}(h_\ell,\gamma_\ell)(k_\ell, \delta_\ell)(h_\ell,\gamma_\ell)^{-1}\\
 &\quad= \left(\sum_{\ell=1}^{m-1} \Delta^{\ell-1}\gamma_\ell k_\ell+\sum_{\ell=1}^{m-1}\Delta^{\ell-1} (1-\delta_\ell)h_\ell,\;\Delta^{m-1}\right) \cdot  ((1- \delta_{m})h_{m}+\gamma_{m} k_{m},\; \delta_{m})
 \\&\quad= \left(\sum_{\ell=1}^{m-1} \Delta^{\ell-1}\gamma_\ell k_\ell+\sum_{\ell=1}^{m-1}\Delta^{\ell-1} (1-\delta_\ell)h_\ell +\Delta^{m-1}\!\cdot ((1- \delta_{m})h_{m} +\gamma_{m} k_{m}),\; \Delta^{m}\right)
 \\&\quad= \left(\sum_{\ell=1}^{m} \Delta^{\ell-1}\gamma_\ell k_\ell+\sum_{\ell=1}^{m}\Delta^{\ell-1} (1-\delta_\ell)h_\ell,\; \Delta^{m})\right).
 \end{align*}
This shows the lemma. 
\qed\end{proof}

The following Lemma is an easy consequence of \cite[Theorem 6.8]{SuzukiGTI81}. 
\begin{lemma}\label{lem: ConjClDn}
The group $D_n$ has the following conjugacy classes: 
   \begin{itemize}
       \item if $n$ is odd: we have $\frac{n+1}{2}$ classes $\{1\}, \{ r^{\pm 1}\}, \ldots, \{r^{\pm\frac{n-1}{2}}\}$ of rotations and the class containing all the reflections $\{r^is\mid i \in \interval{0}{n-1} \}$, 
        \item if$n$ is even:  we have $\frac{n}{2}+1$ classes $\{1\}, \{r^{\frac{n}{2}}\}, \{ r^{\pm 1}\}, \ldots, \{r^{\pm \frac{n}{2}-1}\}$  of rotations. The reflections split into two classes $\{r^{2i}s \mid  i  \in \interval{0}{\frac{n}{2}-1} \}$ and $\{r^{2i+1}s \mid  i \in \interval{0}{\frac{n}{2}-1}\}$. 
   \end{itemize}
\end{lemma}
\begin{lemma} \label{lem: SolDn}
Let $c_1, \ldots, c_k \in D_n$ with $c_i=(a_i,\delta_i)$ as in (\ref{multDn}). 
Then, $(c_1, \ldots, c_k) \in \DPSE{D_n}$ if and only if  $\prod_{i=1}^k \delta_i = 1$ and one of the following conditions hold: 
\begin{enumerate}[(a)]
\item \label{allRot}  $c_i \in \gen{r}$ for all $i$ and there are $\varepsilon_1, \dots, \varepsilon_k \in \{\pm 1\}$ such that $\sum_{i=1}^k \varepsilon_i a_i  = 0$ (modulo $n$). 
\item \label{CaseNOdd} For at least two $i$, $\delta_i = -1$ and, if  $n$ is even,
an even number of the $a_i$ is odd.  
\end{enumerate}
\end{lemma}

\begin{proof}
It is clear that  $\prod_{i=1}^k c_i=(0,1)$ is only possible if  $\prod_{i=1}^k \delta_i = 1$. From now on we assume that this is the case. 

If $c_i \in \gen{r}$, then $\delta_i=1$. Thus, (\ref{allRot}) follows from (\ref{FormSE}). So from now on, we can assume that $\delta_i=-1$ for at least two $i$ (otherwise we are in case (a) or do not have a solution). 

\medskip
 (b)  First consider the case that $n$ is odd. By \cref{lem:reorder} we can assume that $\delta_1=\cdots = \delta_{\ell-1}=1$ and $\delta_{\ell}=\cdots = \delta_k=-1$. Then,  $\prod_{j=1}^{\ell-1}c_j\in \gen{r}$. Moreover, each rotation is the product of two reflections. By \cref{lem: ConjClDn} there exist  $z_{\ell}, z_{\ell+1}\in D_n$ such that $z_{\ell}c_{\ell}z^{-1}_{\ell}z_{\ell+1}c_{\ell+1}z^{-1}_{\ell+1}=(\prod_{j=1}^{\ell-1} c_j)^{-1}$. Note that $k-(\ell+2)$ is even. We can choose $z_j\in D_n$ such that  $z_jc_jz^{-1}_j=c_{j+1}^{-1}$ for $j$ even and $j \in \interval{\ell+2}{k-1} $. If we set $z_j=1$ for the remaining $j$, we get $\prod_{j=1}^k z_jc_jz^{-1}_j=1$.

\medskip
Assume that $n$ is even.  
 By assumption, $\prod_{j=1}^k \delta_j=1$. Taking $(k_\ell,\delta_\ell)=c_\ell$ and $(h_\ell, \gamma_\ell)=z_\ell$, by (\ref{FormSE}) the first component in $\prod_{\ell=1}^k z_\ell c_\ell z^{-1}_\ell$ becomes 
\begin{align}\label{firstComp}
\sum_{\ell=1}^{k} \Delta^{\ell-1}\gamma_\ell a_\ell+\sum_{\delta_\ell=-1}\Delta^{\ell-1}\cdot 2h_\ell.
\end{align}
This sum is even if and only if an even number of the $a_\ell$ is odd. In that case we can choose the $h_i$ such that (\ref{firstComp}) is a multiple of $n$. This shows the lemma. 
\qed\end{proof}

\begin{theorem} \label{thm: DnNPhard}
$\DPSE{(D_n)_{n \in \N}}$, $n$ given in binary, is \NP-complete. 
\end{theorem} 
Note that if we give $n$ in unary notation, then $\DPSE{(D_n)_{n\in\N}}$ can be solved in \P. This follows from \cref{prop:CayNL} because $\abs{D_n}$ is $2n$ and, hence, polynomial.

\begin{proof}
As already mentioned above, we can embed the group $D_n$ into the group $\GL{2}{n}$.  Thus, by \cref{prop:inNP},  $\DPSE{(D_n)_{n\in \N}}$ is in \NP. 
\medskip

To show \NP-hardness, we show that $(a_1,\ldots,a_{k})\in \N^k$ for some $k\in \N$ is a positive instance
of \partition (see \cref{sec:NP}) if and only if $((a_1,1), \ldots, (a_k,1))\in \DPSE{D_n}$ with $n=1+\sum_{i=1}^k a_i$.

First assume that there exists $A \subseteq \interval{1}{k}$ such that $\sum_{i \in A}a_i=\sum_{i  \notin A}a_i$. We choose $\varepsilon_i=1$ for $i \in A$ and $\varepsilon_i=-1$ for $i \notin A$ and set $z_i=(0, \varepsilon_i)$. Then
\begin{align*}
 \prod_{i  = 1}^k z_i(a_i,1)z^{-1}_i=\bigg(\sum_{i  = 1}^k \varepsilon_i a_i,1\bigg)=(0,1)
\end{align*}
and $(0,1)$ is the neutral element in $D_n$. So these $z_i$ are a solution. 
\smallskip

Now assume that $((a_1,1), \ldots, (a_k,1))\in \DPSE{D_n}$. By (\ref{FormSE}) there exist $\gamma_i \in \Z_n$ such that 
$\sum_{i=1}^{k} \gamma_i a_i=r\cdot n$ for some $r \in \Z$.
Let $A=\{i~|~ \gamma_i=1\}$.  Then 
\begin{align*}
\sum_{i \in A}a_i-\sum_{i \notin A}a_i= r\cdot n=r \cdot \Big(1+ \sum_{1\leq i \leq k} a_i\Big).
\end{align*} 
Thus, $r=0$ and $\sum_{i \in A}a_i=\sum_{i \in  \notin A}a_i$. Hence, $(a_1, \ldots, a_k)$ is a positive instance of \partition. 
\qed\end{proof}
We even proved that the following restricted problem is \NP-complete: 
\ynproblem{$n, a_1, \ldots, a_k \in \mathbb{N}$ given in binary.}{Does the equation $\prod_{i=1}^k (a_i,1) $ have a solution in the group $D_n$? }

The groups $D_{2^n}$ are $p$-groups and thus nilpotent.  We can prove the following:

\begin{corollary}
    $\DPSE{(D_{2^n})_{n\in \N}}$  is \NP-complete. In particular, there exists a sequence of nilpotent groups $\cG$ such that $\DPSE{\cG}$ is \NP-hard.
\end{corollary}

\begin{proof} 
Let $\ell$ be the number of bits we need to represent  $\sum_{i=1}^k a_i$. 
We then choose $n=2^\ell$ in the proof of \cref{thm: DnNPhard}.
\qed\end{proof}

\begin{remark}\label{rem: DpNPhard}
Let $p_n$ be the $n$-th prime number. By Cramers conjecture \cite{Cramer36}, $p_{n+1}-p_n \in \Oh((\log p_n)^2)$. Assume that this conjecture is true and let $m=\sum a_i +1$. Then we can find $p \in \prm$,  $p \in \interval{m}{m+(\log(m))^2}$ in polynomial time \cite{Agrawal04}. Choosing $n=p$ in the proof of \cref{thm: DnNPhard} would yield that $\DPSE{(D_p)_{p \in \prm}}$ is \NP-hard.
\end{remark}

\begin{corollary}
The generalized Quaternion group of order $4m$ is defined as
\begin{align*}
Q_{m}=\Gen{x,y}{x^{2m}=1, y^4=1, x^m=y^2, yxy^{-1}=x^{-1}}.
\end{align*}
Then, $\DPSE{(Q_m)_{m \in \N}}$ is \NP-complete. 
\end{corollary}
\begin{proof}
Reduce from \partition as for $D_n$  and use \cref{prop:inNP}.
\qed\end{proof}

\section{Spherical equations for two-by-two matrices}
\label{se:GL2p}

In this section we will consider matrix groups. We first will state some  \NP-hardness results in the \emph{matrix group} input model. Then we will show that $\DPSE{(\TL{2}{p})_{p \in \prm}} \in \P$. After that, as the main result of this section, we  show $\DPSE{(\GL{2}{p})_{p \in \prm}} \in \P$. 
According to \cref{sec: grpZkC} this does not necessarily imply the result for $\TL{2}{p}$.

\subsection{Triangular matrices}

\begin{corollary}\label{lem: DninTwo}
$\DPSE{(\ET{2}{n})_{n \in \N}}$ is \NP-complete.
\end{corollary}
\begin{proof}
We reduce from $\DPSE{(D_n)_{n \in \N}}$. Recall that $D_n\cong \gen{R,S}$ with $R,S \in \ET{2}{n}$ as in \cref{sec: Dihrgrps}.
Matrix multiplication shows that if $A \in \gen{R,S}$ and $B \in  \ET{2}{n}$, $A$ and $B^{-1}AB$ are already conjugated  in $\gen{R,S}$. 
The lemma follows by \cref{thm: DnNPhard} and \cref{prop:inNP}. 
\qed\end{proof}

On the first thought one might think that making the groups of interest bigger, increases the complexity. However, the converse is rather the case as the following result shows. Be aware that in \cref{thm:TinP} we require that $p\in \prm$. We conjecture that this is not necessary and also $\DPSE{(\TL{2}{n})_{n \in \N }} \in \P$.

\begin{theorem}\label{thm:TinP}
$\DPSE{(\TL{2}{p})_{p  \in \prm}}$ is in \P.
\end{theorem}

\newcommand{\prA}[1]{\mathcal{A}_{#1}}
\newcommand{\prC}[2]{\mathcal{C}_{#1,#2}}

\cref{thm:TinP} follows from the observation that the conditions in the following lemma can be checked in polynomial time:

\begin{lemma}\label{lem: DiagTwo}
Let $p \in \prm$ and $C_i=\small\begin{pmatrix}
a_i & b_i\\
0 & c_i
\end{pmatrix}\normalsize \in \TL{2}{p}$ for $i \in \interval{1}{k}$. Then $(C_1, \ldots, C_k)\in \DPSE{\TL{2}{p}}$ if and only if $\prod_{i=1}^{k} a_i=\prod_{i=1}^{k} c_i=1$ and one of the following conditions holds:
\begin{itemize}
\item $a_i\neq c_i$ for some $i$,
\item $ b_i  = 0$ for all $i$, or
\item there are $i\neq \ell$ with $b_i \neq 0  \neq b_\ell $.
\end{itemize}

\end{lemma}

\begin{proof}
In the following we write $X_i=\begin{pmatrix}
x_i & y_i \\
0 & z_i
\end{pmatrix}$  and  $M_k=\prod_{i=1}^k X_iC_iX_i^{-1}$. Further set $\alpha_i=(y_i(c_i-a_i)+x_ib_i)z_i^{-1}$ and $\prA{j}=\prod_{i=1}^{j}a_i$ and $\prC{j}{k}=\prod_{i=j}^{k}c_i$. 
We claim that 
\begin{align}\label{MatrixFormula}
M_k=\begin{pmatrix}
\prA{k} &~ \sum_{j=1}^{k} \prA{j-1}\cdot \alpha_{j} \cdot \prC{j+1}{k}  \\[3mm]
0& \prC{1}{k}
\end{pmatrix}.
\end{align}
First observe that
\begin{align*}
X_iC_iX_i^{-1} &=\begin{pmatrix}
x_i & y_i \\
0& z_i
\end{pmatrix} 
\begin{pmatrix}
a_i & b_i\\
0 & c_i
\end{pmatrix}
\begin{pmatrix}
x_i^{-1} &~ -y_i(x_iz_i)^{-1}\\[0.7mm]
0 & z_i^{-1}
\end{pmatrix}\\
&= \begin{pmatrix}
a_i &\quad (y_i(c_i-a_i)+x_ib_i)z_i^{-1}\\[0.7mm]
0& c_i
\end{pmatrix}
\end{align*}
This shows the claim for $k=1$. Observe that $\cC_{j+1,j}=1$ and $\cC_{j,k}\cdot c_{k+1}=\cC_{j,k+1}$. Thus, by induction we obtain that 
\begin{align*}
M_{k+1}&= 
\begin{pmatrix}
\prA{k+1} &~~\prA{k}\cdot\alpha_{k+1}  + \left( \sum_{j=1}^{k} \prA{j-1}\cdot \alpha_{j}\cdot \prC{j+1}{k}\right)\cdot c_{k+1} \\[2mm]
0&\prC{1}{k+1}
\end{pmatrix}\\[4mm]&
=\begin{pmatrix}
\prA{k+1} & ~ \sum_{j=1}^{k+1} \prA{j-1}\cdot \alpha_{j}\cdot \prC{j+1}{k+1}  \\[2mm]
0&\prC{1}{k+1} 
\end{pmatrix}.
\end{align*}
This shows the claim.
Observe that the diagonal entries of $M_{k}$ are $\prod_{i=1}^{k} a_i$ and $\prod_{i=1}^{k} c_i$. So both have to evaluate to one for a solution to exist.
Thus, $X_1, \dots, X_k$ as above are a solution to the spherical equation $M_k = 1$ if and only if $\prod_{i=1}^{k} a_i=\prod_{i=1}^{k} c_i=1$ and
\begin{align}\label{EqtoSolve}
    \sum_{j=1}^{k} \prA{j-1} \cdot \prC{j+1}{k}\cdot(y_j(c_j-a_j)+x_jb_j)\cdot z_j^{-1}=0.
\end{align}

 Let us assume that $\prod_{i=1}^{k} a_i=\prod_{i=1}^{k} c_i=1$ and distinguish the three cases as in the lemma. First assume that $c_i-a_i \neq 0$ for one $i$. Because $\det(A_i), \det(C_i) \neq 0$ for all $i$, $\prA{j}, \prC{j}{k}$ and $z_j$ are invertible for all $j$. Thus, equation \eqref{EqtoSolve} is equivalent to  \[y_i=  \Big((-z_i^{-1}\cdot\prA{i-1}\cdot\prC{i+1}{k})^{-1} \cdot \sum_{j \neq i} \prA{j-1} \cdot \prC{j+1}{k} \cdot \alpha_{j} -x_ib_i\Big) \cdot (c_i-a_i)^{-1}\]
Now for any choice of $x_j\neq 0, z_j\neq 0$ for all $j$ and  $y_j$ for $j \neq i$, this has a solution for $y_i$. Observe that we do not have any restrictions on $y_i$.\\

Hence, from now on, we can assume that $c_j=a_j$ for all $j$. 
First, assume that $c_j-a_j=b_j=0$ for all $j$. Then the matrices $C_j$ are all scalar and thus in the center and \eqref{EqtoSolve} has a solution. 

Next, consider the case that there is exactly one $i$ with $b_i\neq 0$ and $b_j = 0$ for $j \neq i$.
Then in \eqref{EqtoSolve} the left side is $\prA{i-1} \cdot \prC{i+1}{k}\cdot x_i\cdot b_i\cdot z_i^{-1}$. Because $\prA{i-1},\prC{i+1}{k},b_i,  z_i^{-1}$ are all invertible,  \eqref{EqtoSolve} has no solution with $x_i \neq 0$.

Finally, assume that there are $b_\ell\neq 0$, $ b_h \neq 0$ with $h \neq \ell$. Choose $x_j,z_j \neq 0$ for $j \not\in \{h, \ell\}$ arbitrary and write 
\begin{align*}
    \xi = \sum_{j\not\in \{h, \ell\} } \prA{j-1} \cdot \prC{j+1}{k}\cdot x_j \cdot b_j \cdot z_j^{-1}.
\end{align*}
Then \eqref{EqtoSolve} is equivalent to 
\begin{align*}
    -\prA{h-1} \cdot \prC{h+1}{k}\cdot x_h \cdot b_h \cdot z_h^{-1} -\xi =  \prA{\ell-1} \cdot \prC{\ell+1}{k} \cdot x_\ell \cdot  b_\ell \cdot  z_\ell^{-1}.
\end{align*}
Now, we can make the left side non-zero by choosing an appropriate $x_h \neq 0$. Once the left side is non-zero, we can set 
\begin{align*}
x_\ell=  \left( -\prA{h-1} \cdot \prC{h+1}{k}\cdot x_h \cdot b_h \cdot z_h^{-1} -\xi\right) \cdot(-z_\ell^{-1}\cdot\prA{\ell-1}\cdot\prC{\ell+1}{k})^{-1} \cdot b_\ell^{-1}
\end{align*}
and obtain a solution of \eqref{EqtoSolve}. 
\qed\end{proof}

\subsection{$\DPSE{(\GL{2}{p})_{p \in \prm}}$ is in \P}
\newcommand{\qta}{a}
\newcommand{\rtb}{b}
\newcommand{\xtc}{c}
\newcommand{\ytd}{d}

\newcommand{\atv}{v}
\newcommand{\btx}{x}
\newcommand{\cty}{y}
\newcommand{\dtz}{z}
\newcommand{\ztZ}{Z}

Consider a $2$-by-$2$ matrix
\begin{align*}
A=
\begin{pmatrix}
a& b \\
c& d \\
\end{pmatrix}
\end{align*}
with entries $a,b,c,d\in\MZ_p$. By definition, its \emph{determinant}
$\det(A)$ is $ad-bc$ and its \emph{trace} $\tr(A)$ is $a+d$.
Recall that $A$ is invertible if and only if $\det(A)\ne 0$ and, if this is the case, its inverse is
\[
A^{-1}=
\frac{1}{\det(A)}
\begin{pmatrix}
d& -b \\
-c& a \\
\end{pmatrix}.
\]
The determinant and the trace of $A$ are invariants of
the conjugacy class of $A$, i.e., conjugation of $A$ preserves their values.
The \emph{characteristic polynomial} for $A$ is defined by
\begin{align*}
p_A(\lambda)=
\det(A-\lambda I)  = 
\left|
\begin{array}{cc}
a-\lambda& b \\
c& d-\lambda \\
\end{array}
\right| 
=\lambda^2-\tr(A)\lambda+\det(A)
\end{align*}
and is an invariant for the conjugacy class of  $A$. Its zeros $\lambda_1,\lambda_2$
are called the \emph{eigenvalues} for $A$.
If $p_A(\lambda)$ is irreducible, then 
$\lambda_1,\lambda_2\in \MZ_p[\sqrt{(\tr(A))^2-4\det(A)}] = \Z[\xi]$ for any quadratic non-residue $\xi \in \Z_p$.
Otherwise, we have $\lambda_1,\lambda_2\in\MZ_p$.
It is a common knowledge that $\tr(A)=\lambda_1+\lambda_2$
and $\det(A)=\lambda_1\lambda_2$.

Let $K$ be a field extension of $\Z_p$. We say that matrices $A,B\in\GL{2}{K}$ are \emph{similar} if
$A=X^{-1}BX$ for some $X\in \GL{2}{\overline{K}}$ where $\overline{K}$ is the algebraic closure of $K$.
Every $A\in\GL{2}{p}$ 
with the eigenvalues $\lambda_1,\lambda_2$ in a field extension $K$ is similar to
one of the following matrices:
\small
\begin{align*}
\begin{pmatrix}
\lambda_1 & 0 \\
0& \lambda_2 \\
\end{pmatrix}
\mbox{ or }
\begin{pmatrix}
\lambda & 1 \\
0& \lambda \\
\end{pmatrix}
\end{align*}
\normalsize
where in the second case $\lambda=\lambda_1=\lambda_2$.
Similarity of matrices does not depend on the base field, 
i.e., if $L$ is an extension of $K$, and $A$ and $B$ are two matrices over $K$, then $A$ and $B$ are conjugate as matrices over $K$ if and only if they are conjugate as matrices over $L$ 
(because the rational canonical form over $K$ is also 
the rational canonical form over $L$).
 This means that one may use Jordan forms that only exist over a larger field to determine whether the given matrices are similar.

In our analysis we distinguish matrices of four types (refining the two cases from above). We say that $A$ is
\begin{itemize}
\item 
\emph{scalar} if it is diagonal 
and $\lambda_1=\lambda_2\in \Z_p$.
\item 
of type-I if 
$\lambda_1,\lambda_2\in \Z_p$ and $\lambda_1\ne\lambda_2$. 
In this case $A$ is similar to 
\small
\begin{align*}
\begin{pmatrix}
\lambda_1 & 0 \\
0& \lambda_2 \\
\end{pmatrix}.
\end{align*}
\normalsize
\item 
of type-II if 
$\lambda_1,\lambda_2\notin \Z_p$ and $\lambda_1\ne\lambda_2$. 
In this case $A$ is similar to 
\small
\begin{align*}
\begin{pmatrix}
s & t \\
1& s \\
\end{pmatrix}
\end{align*}
where $s,t \in \Z_p$, $t$ is not a quadratic residue, $\det(A)=s^2-t$, and $\tr(A)=2s$.
\item of type-III if  $A$ is not scalar,
$\lambda_1,\lambda_2\in \Z_p$ and $\lambda_1=\lambda_2$. In this case  $A$ is similar to 
\small
\begin{align*}
\begin{pmatrix}
\lambda_1 & 1 \\
0& \lambda_1 \\
\end{pmatrix}.
\end{align*}
\normalsize
\end{itemize}
Every $A\in \GL{2}{p}$ classifies as a matrix of
one type defined above. Note that, if $A$ is non-scalar, its conjugacy class is uniquely defined by its trace and determinant (however, scalar matrices can have the same trace and determinant as type-III matrices).
For $S\subseteq \GL{2}{p}$ define
\begin{align*}
    \tr(S) = \Set{\tr(A)}{A\in S}.
\end{align*}
Notice that, because the trace is invariant under conjugation, for any $A,B\in\GL{2}{p}$ we have
\begin{align*}
\tr(\Set{A^YB^Z}{Y,Z\in\GL{2}{p}}) = \tr(\Set{AB^Z}{Z\in\GL{2}{p}}).
\end{align*}
 \normalsize  For matrices $A,B$ we write \[T(A,B)=\tr(\Set{AB^Z}{Z\in\GL{2}{p})}.\] 
Before we investigate $T(A,B)$ further, let us recall some tools from computational number theory.

\subsubsection{Number theory basics.}
We will use the following two results from computational number theory.
Computing the Legendre/Jacobi symbol is a standard result, see e.g.\ \cite[Theorem 83]{hardy75}:

\begin{lemma}\label{lem:computeJacobi}
    There is a polynomial-time algorithm that given $a \in \Z$ and a prime $p$ decides whether $a$ is a square modulo $p$.
\end{lemma}

\begin{theorem}[Tonelli-Shanks algorithm, see \cite{edam16}]\label{thm:tonelli}
There is an algorithm that finds a solution to $x^2\equiv c \bmod p$
in random polynomial time.
\end{theorem}

\begin{theorem}[L. Adleman, D. Estes, K. McCurley, \cite{Adleman-Estes-McCurley:1987}]
\label{th:bivariate}
There exists an algorithm that, upon input $k$, $m$, and $n$
with odd $n$ and $\gcd(km,n) = 1$, will output a solution to the congruence 
\begin{equation}\label{eq:bivariate}
x^2-ky^2 \equiv m \bmod n
\end{equation}
with probability $1 - \varepsilon$ in 
$O(\log(\varepsilon^{-1}\log|k|)\log^3(n) \log|k|)$ steps. 
\end{theorem}

We require the following lemma.

\begin{lemma}\label{le:type2function}
Suppose that $p\ge 3$.
Then for all $k\in \MZ_p$ and all quadratic non-residues $t,\rtb\in \MZ_p$, there exist $\btx \in \MZ_p$ and $u\ne 0$ satisfying
\begin{equation}\label{eq:type2-main-eq}
u \rtb + \tfrac{1}{u} t - \tfrac{1}{u} \btx^2 = k.
\end{equation}
Furthermore, given $k$, $t$, and $b$, the
required $\btx$ and $u$ can be found 
in random polynomial time.
\end{lemma} 

\begin{proof}
We claim that for some choice of $\btx$ the quadratic equation
\begin{align*}
\rtb u^2 - ku + (t-\btx^2) = 0
\end{align*}
has a nontrivial solution of the form 
$u=\tfrac{1}{2\rtb}\big(k\pm\sqrt{k^2-4\rtb(t-\btx^2)}\big)$.
Indeed, the expression $4\rtb(t-\btx^2)$ assumes $(p+1)/2$ values. 
Hence, there exists a value for $\btx$ such that the discriminant 
$k^2-4\rtb(t-\btx^2)$ is a square.
By \cref{th:bivariate} we can find such value in random polynomial time.

Now, if $k^2-4\rtb(t-\btx^2)\ne 0$, then our choice of $\btx$ gives a nontrivial solution $u$
(since $p\ne 2$).
If $k^2-4\rtb(t-\btx^2)=0$, then the solution is unique. It is zero
if and only if $k=0$, in which case $-4\rtb(t-\btx^2)=0$.
That is impossible because $\rtb,t$ are quadratic non-residues modulo $p$.
\qed\end{proof}

\subsubsection{Conjugacy and type-finding.}

\begin{lemma}\label{prop:typefinding}
    The following problems are in polynomial time:
    \begin{enumerate}[(a)]
\item Given $A\in\GL{2}{p}$, compute its trace, determinant and characteristic polynomial.
    
 \item Given $A,B\in\GL{2}{p}$, decide whether they are conjugate.

        \item Given $A\in\GL{2}{p}$, compute its type. 

        If $A$ is of type-II, compute a matrix $\begin{pmatrix}
s & t \\
1 & s \\
\end{pmatrix}$ that is similar to $A$.

        If $A$ is of type-III, compute a matrix $\begin{pmatrix}
s & 1 \\
0 & s \\
\end{pmatrix}$ that is similar to $A$.
     
    \end{enumerate}
    Moreover, if $A\in\GL{2}{p}$ is of type-I, we can compute in random polynomial time a conjugate matrix of the form $\begin{pmatrix}
s & 0 \\
0 & t \\
\end{pmatrix}$.

Finally, given conjugate matrices $A,B\in\GL{2}{p}$, one can compute a matrix $Z \in \GL{2}{p}$ with $A = B^Z$ in randomized polynomial time.
\end{lemma}

\begin{proof}
Part (a) is clear since we only need to perform a bounded number of arithmetic operations. 
 To see (b), recall that two non-scalar 2-by-2 matrices $A,B$ are conjugate if and only if their characteristic polynomials agree. 

(c): Without loss of generality $A$ is non-scalar. To compute the type of  $A$, we need to compute the discriminant $\xi = (\tr(A))^2-4\det(A)$ of the characteristic polynomial, which can be done in polynomial time. By \cref{lem:computeJacobi}, we can decide whether $\xi$ is a quadratic residue.

If $\xi$ is a non-zero quadratic residue, we know that $A$ is of type-I.
If $\xi$ is a quadratic non-residue, we know that $A$ is of type-II and $A$ is similar to a matrix $\begin{pmatrix}
s & t \\
1& s \\
\end{pmatrix}$. We can compute $s$ and $t$ via $s = \tr(A)/2$ and $t = s^2 - \det(A)$.
 Now, if $\xi = 0$, we know that $A$ is of type-III and $A$ is similar to  $\begin{pmatrix}s & 1 \\
0& s \\
\end{pmatrix}$ where $s= \tr(A)/2$.

If $A$ is of type-I, by \cref{thm:tonelli} we can compute the square root of the discriminant in random polynomial time and, hence, compute the required conjugate matrix.

Finally, let $A,B\in\GL{2}{p}$ be conjugate matrices. We can compute their eigenvalues $\lambda_1,\lambda_2 \in \Z_p[\xi]$ in random polynomial time by \cref{thm:tonelli}. Then, we solve the system of linear equations $(A - \lambda_iI)x = 0$ (for type-III matrices, i.e.\ $\lambda_1=\lambda_2$, we also solve  $(A - \lambda_1 I)^2x = 0$) to compute a basis of (generalized) eigenvectors. These will form a conjugating matrix. Note that for type-II matrices we might get eigenvectors in the field extension $\Z_p[\xi]$. If this is the case, we multiply them by a scalar from $\Z_p[\xi]$ to bring them into the base field $\Z_p$.
\qed
\end{proof}

In the following paragraphs we will investigate $T(A,B)$. We distinguish three cases according to the type of $B$ (the case that one matrix is scalar is not of interest). We will write \small$Z=\begin{pmatrix}
\atv& \btx \\
\cty& \dtz \\
\end{pmatrix}$.

\subsubsection{Type--I.}
Consider matrices
\small
\begin{align*}
A=
\begin{pmatrix}
\qta& \rtb \\
\xtc& \ytd \\
\end{pmatrix}
,\qquad 
B=
\begin{pmatrix}
s& 0 \\
0& t \\
\end{pmatrix}
\in\GL{2}{p}.
\end{align*}
\normalsize

\begin{lemma}\label{pr:type1trace} 
If $A$ and $B$ are non-scalar (i.e.\ $B$ of type-I), then $T(A,B)=\MZ_p$. 
Furthermore, given $A$ and $B$ as well as $k\in \MZ_p$ one can find $\ztZ\in\GL{2}{p}$ satisfying $\tr(AB^\ztZ)=k$ in polynomial time.
\end{lemma}

\begin{proof}
Assume that $\atv\dtz-\btx\cty=1$ and write $M = AB^\ztZ$. Then \begin{align*}
M
&=
\begin{pmatrix}
\qta& \rtb \\
\xtc& \ytd \\
\end{pmatrix}
\begin{pmatrix}
\dtz& -\btx \\
-\cty& \atv \\
\end{pmatrix}
\begin{pmatrix}
s& 0 \\
0& t \\
\end{pmatrix}
\begin{pmatrix}
\atv& \btx \\
\cty& \dtz \\
\end{pmatrix}
\\[1mm]
&=
\begin{pmatrix}
\cty t (\atv \rtb - \btx \qta ) + \atv s (\dtz \qta - \cty \rtb ) &~ \dtz t (\atv \rtb - \btx \qta ) + \btx s (\dtz \qta - \cty \rtb ) \\
\cty t (\atv \ytd - \btx \xtc ) + \atv s (\dtz \xtc - \cty \ytd) &~ \dtz t (\atv \ytd - \btx \xtc ) + \btx s (\dtz \xtc - \cty \ytd) \\
\end{pmatrix}.
\end{align*}

Its trace can be simplified as follows:
\begin{align*}
\tr(M)
&\ \ =\ \ 
-\atv \cty \rtb s + \atv \cty \rtb t + \atv \dtz \qta s + \atv \dtz t \ytd - \btx \cty \qta t - \btx \cty s \ytd + \btx \dtz s \xtc - \btx \dtz t \xtc\\
&\ \ =\ \ 
-\atv \cty \rtb s + \atv \cty \rtb t + \btx \cty \qta s - \btx \cty \qta t - \btx \cty s \ytd + \btx \cty t \ytd + \btx \dtz s \xtc - \btx \dtz t \xtc + \qta s + t \ytd.
\end{align*}
Hence, we have
\begin{align*}
&T(A,B)\\&=\Set{\atv \cty (- \rtb s + \rtb t) + 
\btx \cty (\qta s - \qta t - s \ytd + t \ytd) + 
\btx \dtz (s \xtc - t \xtc) + 
\qta s + t \ytd}{\atv\dtz-\btx\cty=1}.
\end{align*}
Consider the following conditions:
\begin{align}
\tag{C1}
-\rtb s+\rtb t&=(t-s)\rtb =0\\
\tag{C2}
\qta s - \qta t - s \ytd + t \ytd&=(s-t)(\qta-\ytd)= 0\\
\tag{C3}
s \xtc - t \xtc&=(s-t)\xtc =0.
\end{align}
It is easy to see that 
\begin{align*}
\mbox{(C1) and (C2) and (C3) are satisfied}
&\ \iff\ \ 
s=t\mbox{ or }\rtb=\xtc=\qta-\ytd=0\\
&\ \iff\ \ 
\mbox{$B$ or $A$ is scalar.}
\end{align*}

Since $A$ and $B$ are non-scalar,
at least one condition is not satisfied.
If (C1) is not satisfied, then $-\rtb s+\rtb t\ne 0$ and setting $\btx=0$ we get
\begin{align*}
T \supseteq
\Set{\atv \cty (- \rtb s + \rtb t) + 
\qta s + t \ytd}{\atv \dtz=1} = \MZ_p.
\end{align*}
If (C3) is not satisfied, then $s\xtc-t\xtc\ne 0$ and setting $\cty=0$ we get
\begin{align*}
T \supseteq
\Set{\btx \dtz (s \xtc - t \xtc) + 
\qta s + t \ytd}{\atv \dtz=1} = \MZ_p.
\end{align*}
If (C2) is not satisfied, then $\qta s - \qta t - s \ytd + t \ytd\ne 0$.
Notice that now we may assume that (C1) and (C3) are satisfied.
Hence,
\begin{align*}
T \supseteq
\Set{\btx \cty (\qta s - \qta t - s \ytd + t \ytd) + 
\qta s + t \ytd}{\atv\dtz-\btx\cty=1}=\MZ_p.
\end{align*}
Thus, $T=\MZ_p$ as claimed.
Obviously, in each case considered above
the entries  $\atv,\btx,\cty,\dtz\in\MZ_p$ satisfying
$\tr(AB^\ztZ)=k$ can be found in polynomial time.\qed
\end{proof}

\subsubsection{Type--II.}

Consider two matrices of type-II
\small
\begin{align*}
A=
\begin{pmatrix}
\qta & \rtb \\
1 & \qta \\
\end{pmatrix}
,\qquad
B=
\begin{pmatrix}
s & t \\
1 & s \\
\end{pmatrix}
\in \GL{2}{p},
\end{align*}
\normalsize
where $t$ and $\rtb$ are quadratic non-residues.

\begin{lemma}\label{pr:type2trace}
For matrices $A,B$ of type-II we have $T(A,B)= \MZ_p$.
Furthermore, given $A$ and $B$ as well as $k\in \MZ_p$ one can find a matrix $\ztZ\in\GL{2}{p}$ satisfying $\tr(AB^\ztZ)=k$ in polynomial time.
\end{lemma}

\begin{proof}
Let $\delta=\atv\dtz-\btx\cty$ and notice that
\begin{align*}
AB^\ztZ=
\begin{pmatrix}
\qta& \rtb \\
1& \qta \\
\end{pmatrix}
\cdot
\frac{1}{\delta}
\begin{pmatrix}
\dtz& -\btx \\
-\cty& \atv \\
\end{pmatrix}
\begin{pmatrix}
s& t \\
1& s \\
\end{pmatrix}
\begin{pmatrix}
\atv& \btx \\
\cty& \dtz \\
\end{pmatrix}
\end{align*}
which trace is $\tfrac{1}{\delta}(\atv^2 \rtb - \btx^2 - \cty^2 \rtb t + \dtz^2 t) + 2 \qta s.$
Therefore,
\begin{align*}
T(A,B)&=\Set{\tfrac{1}{\delta}(\atv^2 \rtb - \btx^2 - \cty^2 \rtb t + \dtz^2 t) + 2 \qta s}{\delta\ne 0, \btx \in \Z_p}\\
&\supseteq
\Set{\tfrac{1}{\atv \dtz}(\atv^2 \rtb - \btx^2 + \dtz^2 t) + 2 \qta s}{\atv,\dtz\ne 0,\btx \in \Z_p}
\tag{set $\cty=0$}\\
&=
\Set{\tfrac{\atv}{\dtz} \rtb + \tfrac{\dtz}{\atv} t - \tfrac{\dtz}{\atv}\btx^2 + 2 \qta s}{\atv,\dtz\ne 0,\btx \in \Z_p}&&\tag{replace $\btx$ by $\dtz \btx$}
\end{align*} 
Thus, by \cref{le:type2function}, $T(A,B)=\Z_p$.\qed
\end{proof}

\subsubsection{Type III.}

Consider matrices
\small
\begin{align*}
A=
\begin{pmatrix}
\qta& \rtb \\
\xtc& \ytd \\
\end{pmatrix}
\mbox{ and }
B=
\begin{pmatrix}
s& 1 \\
0& s \\
\end{pmatrix} \in \GL{2}{p}. 
\end{align*}
\normalsize

\begin{lemma}\label{pr:type3trace}
Suppose that $A$ is not scalar.  The following hold:

\begin{enumerate}[(a)] 
\item
$\MZ_p \setminus \{s(\qta+\ytd)\} \subseteq T(A,B)$.
\item
It can be decided in polynomial time whether
$s(\qta+\ytd) \in T(A,B)$.
\item Given $A$ and $B$ as well as $k\in \MZ_p$, a matrix $\ztZ\in\GL{2}{p}$ with $\tr(AB^\ztZ)=k$
 can be found in random polynomial time.
\end{enumerate}
\end{lemma}

\begin{proof}
Because $\atv\dtz-\btx\cty=\delta\ne 0$, we have
\begin{align*}
AB^\ztZ&=
\begin{pmatrix}
\qta& \rtb \\
\xtc& \ytd \\
\end{pmatrix}
\cdot
\frac{1}{\delta}
\begin{pmatrix}
\dtz& -\btx \\
-\cty& \atv \\
\end{pmatrix}
\begin{pmatrix}
s& 1 \\
0& s \\
\end{pmatrix}
\begin{pmatrix}
\atv& \btx \\
\cty& \dtz \\
\end{pmatrix}\\[1mm]
&=
\frac{1}{\delta}
\begin{pmatrix}\small
\cty (s (\atv \rtb - \btx \qta) - \cty \rtb + \dtz \qta) + \atv s (\dtz \qta - \cty \rtb) &~ \dtz (s (\atv \rtb - \btx \qta) - \cty \rtb + \dtz \qta) + \btx s (\dtz \qta - \cty \rtb) \\
\cty (s (\atv \ytd - \btx \xtc) - \cty \ytd + \dtz \xtc) + \atv s (\dtz \xtc - \cty \ytd) &~ \dtz (s (\atv \ytd - \btx \xtc) - \cty \ytd + \dtz \xtc) + \btx s (\dtz \xtc - \cty \ytd)\\
\end{pmatrix}\\[1mm]
\shortintertext{which trace is}
\tfrac{1}{\delta}&\left(\cty (s (\atv \rtb - \btx \qta) - \cty \rtb + \dtz \qta) + \dtz (s (\atv \ytd - \btx \xtc) - \cty \ytd + \dtz \xtc) + \atv s (\dtz \qta - \cty \rtb) + \btx s (\dtz \xtc - \cty \ytd)\right)\\
&=
\tfrac{1}{\delta} \left((\atv \dtz - \btx \cty) \qta s + (\atv \dtz   - \btx \cty) s \ytd + \cty^2 (-\rtb) + \cty \dtz (\qta -  \ytd) + \dtz^2 \xtc\right)\\
&=\tfrac{1}{\delta} \left(\delta \qta s + \delta s \ytd+ \cty^2 (-\rtb) +\cty \dtz (\qta -  \ytd) + \dtz^2 \xtc \right)\\
&=\tfrac{1}{\delta} \left(\cty^2 (-\rtb) +\cty \dtz (\qta -  \ytd) + \dtz^2 \xtc \right) + \qta s + s \ytd.
\end{align*}
Therefore,
\begin{align*}
T(A,B)
&=
\Set{\tfrac{1}{\delta} \left(\cty^2 (-\rtb) +\cty \dtz (\qta -  \ytd) + \dtz^2 \xtc \right) + \qta s + s \ytd}{ \delta=\atv\dtz-\btx\cty \neq 0}\\
&=
\Set{\tfrac{1}{\delta} \left(\cty^2 (-\rtb) +\cty \dtz (\qta -  \ytd) + \dtz^2 \xtc \right) + \qta s + s \ytd}{\mbox{$\cty \neq 0$ or $\dtz\neq 0$ }}.
\end{align*}
If $A$ is not scalar, then $\cty^2 (-\rtb) +\cty \dtz (\qta -  \ytd) + \dtz^2 \xtc$ 
assumes a nontrivial value for some (not both trivial) $\cty$ and $\dtz$.
Varying $\atv$ and $\btx$ we can get any nontrivial value for
$\delta$ without changing $\cty^2 (-\rtb) +\cty \dtz (\qta -  \ytd) + \dtz^2 \xtc$
and attain all values in $\MZ_p\setminus\{s(\qta+\ytd)\}$.
Therefore, (a) holds.
Finally, 
\begin{align*}
s(\qta+\ytd) \in  T(A,B)
&\ \Leftrightarrow\ 
\cty^2 (-\rtb) +\cty \dtz (\qta -  \ytd) + \dtz^2 \xtc=0
\mbox{ has a solution with $\cty\ne 0$ or $\dtz\ne 0$}\\
&\ \Leftrightarrow\ 
(\qta-\ytd)^2+4\rtb \xtc 
\mbox{ is a quadratic residue of }p,
\end{align*}
which can be checked in polynomial time (\cref{lem:computeJacobi}).
To see the last equivalence we can multiply by $y^{-1}$ or $z^{-1}$ to obtain a polynomial that has $(\qta-\ytd)^2+4\rtb \xtc$ as discriminant.   Thus, (b) holds. 

In order to compute the matrix $Z$, we distinguish the same two cases as above: if $k\neq s(a+d)$, we can find some $y,z$ such that $\cty^2 (-\rtb) +\cty \dtz (\qta -  \ytd) + \dtz^2 \xtc \neq 0$ with a constant number of tries, set $v=1$, and then solve a linear equation.
If $k= s(a+d)$, we need to solve the quadratic equation as above, which can be done in random polynomial time by \cref{thm:tonelli} (and we can choose arbitrary values for $x$ and $v$).\qed 
\end{proof}
Now, consider a special case when $A$ and $B$ are of type--III. In this case,
\begin{align*}\label{eq:type3type3}
T(A,B) =
\Set{-\tfrac{\cty^2}{\delta} + 2\qta s}{\cty\in\MZ_p,\ \delta\ne 0} = 
\MZ_p.
\end{align*}
So, if the matrix on the right hand side is of type-1 or type-2, then the equation $AB^Z=C^Y$ always has a solution because the conjugacy class is uniquely determined by $\det(C)$ and $\tr(C)$. Thus it remains to consider $C$ of type-3.

\begin{lemma}\label{pr:eq-3type3-b}
Suppose that $p\ge 3$, $\qta\ne 0$, $s\ne 0$. 
Then the equation 
\begin{equation}\label{eq:eq-3type3-b}
\small
\begin{pmatrix}
\qta& 1 \\
0& \qta \\
\end{pmatrix}
\cdot
\begin{pmatrix}
s& 1 \\
0& s \\
\end{pmatrix}
^{\ztZ_2}
=
\begin{pmatrix}
-\qta s& 1 \\
0& -\qta s \\
\end{pmatrix}
^{\ztZ_3}
\end{equation}
\normalsize
has a solution for $\ztZ_2,\ztZ_3 \in \GL{2}{p}$.
Furthermore, $\ztZ_2$ and $\ztZ_3$ satisfying \eqref{eq:eq-3type3-b}
can be found in random polynomial time.
\end{lemma}

\begin{proof}
Simplify the left-hand side of the equation to get 
\small
\begin{align} \label{CalcTwoTypIII}
\begin{pmatrix}
\qta& 1 \\
0& \qta \\
\end{pmatrix}
\cdot
\frac{1}{\delta}
\begin{pmatrix}
\dtz& -\btx \\
-\cty& \atv \\
\end{pmatrix}
\begin{pmatrix}
s& 1 \\
0& s \\
\end{pmatrix}
\begin{pmatrix}
\atv& \btx \\
\cty& \dtz \\
\end{pmatrix}
=
\frac{1}{\delta}
\begin{pmatrix}
-\cty^2 + \cty \dtz \qta + \delta \qta s &~ -\cty \dtz + \dtz^2 \qta + \delta s\\
-\cty^2 \qta & \delta \qta s - \cty \dtz \qta\\
\end{pmatrix}
\end{align}
\normalsize
where $\delta=\atv\dtz-\btx\cty$.
The trace of the obtained matrix is $2 \qta s-\tfrac{\cty^2}{\delta}$. So for (\ref{eq:eq-3type3-b}) to hold we need that  $
\cty^2 = 4\qta s\delta.$
Fix any $\cty\neq 0$ and choose $\atv,\btx,\dtz$ satisfying
$\cty^2 = 4\qta s \delta$.
Observe that the right side in (\ref{CalcTwoTypIII}) is not scalar because its lower-left entry is $-\tfrac{1}{\delta}\cty^2\qta  = -4\qta ^2s\neq 0$. By the above choice of $\cty$, its determinant is $ \qta^2s^2$ and its trace is  $-2\qta s$.
So the eigenvalues are $\lambda_1=\lambda_2=-\qta s$ and thus the right hand side in (\ref{CalcTwoTypIII}) is of type III. The choice of $\atv,\btx,\cty,\dtz$ fix a matrix $\ztZ_2$.
Hence, $\ztZ_3$ can be found in random polynomial time by \cref{prop:typefinding}.
\qed\end{proof}

\begin{lemma}\label{pr:eq-3type3-a}
Let $p \geq 3$. Suppose that  $\qta\ne 0$, $s\ne 0$.
Then the equation 
\small
\begin{equation}\label{eq:eq-3type3-a}
\begin{pmatrix}
\qta& 1 \\
0& \qta \\
\end{pmatrix}
\cdot
\begin{pmatrix}
s& 1 \\
0& s \\
\end{pmatrix}
^{Z_2}
=\begin{pmatrix}
\qta s& 1 \\
0& \qta s \\
\end{pmatrix}
^{Z_3}
\end{equation}
\normalsize
has a solution for $Z_2,Z_3 \in \GL{2}{p}$.
Furthermore, $\ztZ_2$ and $\ztZ_3$ satisfying \eqref{eq:eq-3type3-a}
can be found in random polynomial time.
\end{lemma}
\begin{proof}
The trace of the right-hand side matrix is $2\qta s$. Hence,
\begin{align*}
    -\tfrac{\cty^2}{\delta}+2\qta s=2\qta s. 
\end{align*}
Therefore, $\cty=0$ and on the left-hand side we get
\small
\begin{align*}
\begin{pmatrix}
\qta & 1 \\
0& \qta  \\
\end{pmatrix}
\frac{1}{\delta}\cdot
\begin{pmatrix}
\dtz& -\btx \\
0& \atv \\
\end{pmatrix}
\begin{pmatrix}
s& 1 \\
0& s \\
\end{pmatrix}
\begin{pmatrix}
\atv& \btx \\
0& \dtz \\
\end{pmatrix}
=
\frac{1}{\atv \dtz}\begin{pmatrix}
\atv \dtz \qta s&~ \dtz^2 \qta + \atv \dtz s\\
0& \atv \dtz \qta s\\
\end{pmatrix}
=
\begin{pmatrix}
\qta s&~~ \frac{\dtz}{\atv}\cdot \qta +s\\
0& \qta s\\
\end{pmatrix}
\end{align*}
\normalsize
Hence, the equation \eqref{eq:eq-3type3-a} 
has a solution if and only if
$
 \frac{\dtz}{\atv}\cdot \qta +s \ne 0
$
for some $\dtz\ne 0$ (otherwise, the right hand side is scalar and thus not of type-III). We can choose $\atv \neq 0$ arbitrarily and then $\dtz$ such that $\dtz \ne -\tfrac{s}{\atv \qta}$. 
Such $\dtz$ can be found in polynomial time.
This defines a matrix $\ztZ_2$. Hence, $\ztZ_3$ can be found in random polynomial time by \cref{prop:typefinding}.\qed
\end{proof}

\newcommand{\ctC}{C}
\begin{theorem}\label{thm: GLinP}
$\DPSE{(\GL{2}{p})_{p \in \prm}} \in  \P$. 
Furthermore, a solution (if it exists) can be found in random 
polynomial time.
\end{theorem}

\begin{proof}
If $p=2,3$,  the theorem follows from \cref{thm: finitegroup}. Thus, from now on we assume $p\geq 5$.
Let $\ctC_1,\dots, \ctC_k\in \GL{2}{p}$. If one of the matrices, say $\ctC_k$, is scalar, then
 $(\ctC_1,\dots, \ctC_k)\in \DPSE{\GL{2}{p}}$ if and only if $(\ctC_1,\dots, \ctC_{k-1}\ctC_k)\in \DPSE{\GL{2}{p}}$.
 Hence, we may assume that 
$\ctC_1,\dots,\ctC_k$ are non-scalar matrices.
Furthermore, we may assume that 
$\det(\ctC_1\cdots \ctC_k)=1$ (otherwise $(\ctC_1,\dots, \ctC_k)\notin \DPSE{\GL{2}{p}}$). Note that, since $\det(\ctC_1\cdots \ctC_k) = \det(\ctC_1)\cdots \det(\ctC_k)$, we can check this in polynomial time.

We consider the cases $k=1,2,3$ and $k \geq 4$.
If $k=1$, then $(\ctC_1) \in \DPSE{\GL{2}{p}}$ if and only if $\ctC_1=1$.
If $k=2$  the problem is equivalent
to the conjugacy of $\ctC_1$ and $\ctC_2^{-1}$, which is decidable 
in polynomial time for $\GL{2}{p}$ (see \cref{prop:typefinding}).

Let $k=3$ and assume that all matrices are of type-I or type-II. By \cref{prop:typefinding} we can check this in polynomial time. 
We write the equation as $\ctC_1\ctC_2^{\ztZ_2}=(\ctC_3^{-1})^{\ztZ_3}$. 
The conjugacy class of $\ctC_3^{-1}$ is uniquely defined
by $\det(\ctC_3^{-1})$ and $\tr(\ctC_3^{-1})$. Hence, there exists a solution if and only if
\begin{align*}
\tr(\ctC_3^{-1})\ \in\ \tr\left(\Set{\ctC_1 \ctC_2^{\ztZ_2}}{\ztZ_2\in\GL{2}{p}}\right).
\end{align*}

\begin{itemize}
\item 
If either  $\ctC_1$ or $\ctC_2$, say $\ctC_2$, is of type-I, then,
by \cref{pr:type1trace}, $T(\ctC_1,\ctC_2)=\MZ_p$ 
and the equation has a solution.
\item 
If both $\ctC_1$ and $\ctC_2$ are of type-II, then, by \cref{pr:type2trace}, $T(\ctC_1,\ctC_2)=\MZ_p$
and the equation has a solution.
\end{itemize}
If one or two, but not all three matrices are of type-III let $\ctC_3$ be the matrix not of type-III (due to \cref{lem:reorder}). Then we use \cref{pr:type3trace} to decide if there exists a solution. For this we only need to know the trace of the matrices. Note that the conjugacy class of $\ctC_3^{-1}$ is again uniquely defined
by $\det(\ctC_3^{-1})$ and $\tr(\ctC_3^{-1})$.   
If $\ctC_1,\ctC_2,\ctC_3$ are all of type-III, we are in one of the cases   in \cref{pr:eq-3type3-b} or \cref{pr:eq-3type3-a} (otherwise, the determinant and the eigenvalues do not fit). So, we have a solution.

So, to summarize, when $k=3$,
the equation may have no solution only if at least one but not more than two of the matrices are of type-III. Using \cref{prop:typefinding} we can determine the type of the matrices in polynomial time and thus decide in which case we are.

Now consider the case $k \geq 4$. We first subsequently multiply two matrices together until we are left with only three matrices $\ctC_1,\ctC_2,\ctC_3$.
If none or all of them are of type-III, we proceed as above to show that there is a solution.
If at least one but at most two of the matrices are of type-III, let $\ctC_3$ be the matrix not of type-III.
By \cref{pr:type3trace}, 
$T(\ctC_1,\ctC_2) \supseteq  \Z_p \setminus \{d\}$ for $d=\tr(\ctC_1)\cdot \tr(\ctC_2)/2$ if $\ctC_2$ is of type-III.
So we are left with the case that $\tr(\ctC_3)=d$. For some $i \in \{1,2,3\}$ there exist $D,E \in \GL{2}{p}$ such that $\ctC_i=DE$.
If $i=3$, by Propositions \ref{pr:type1trace}-\ref{pr:eq-3type3-b} and because $p \geq 5$ we can choose $U \in \GL{2}{p}$ such that $\tr(DE^U)\neq d$  and $(\tr(DE^U))^2 \neq 4 \det(DE^U)$  (and thus $DE^U$ not of type-III). 
If $i=1,2$, again by the above propositions we choose $U \in \GL{2}{p}$ such that $\tr(DE^U) \neq \tr(\ctC_i)$. By \cref{pr:type3trace}, $d \in T(DE^U,C_{3-i})$. 

In each case, by \cref{prop:typefinding} and Lemmas \ref{pr:type1trace}-\ref{pr:eq-3type3-b} it follows that we can find a solution in random polynomial time if one exists. Note that we need to calculate the normal form of the matrix only for calculating the solutions. For the decision problem it suffices to determine the type. 
\qed

\end{proof}

Notice that the proof also shows that if $k \geq 4$ we always have a solution if $\det(C_1)\cdots \det(C_k)=1$.

\begin{remark} \label{rem: CramerforET}  
If Cramer's conjecture is true \cite{Cramer36}, by \cref{rem: DpNPhard},  $\DPSE{(\ET{2}{p})_{p \in \prm}}$ is \NP-hard. Observe that $\DPSE{(\UT{2}{p})_{p \in \prm}}$ is in \P because $\UT{2}{p}$ is abelian and $\UT{2}{p} \leq \ET{2}{p} \leq \GL{2}{p}$.  In \cref{sec: grpZkC}  we show that   there are indeed sequences of groups 
 $\mathcal{G'}=(G'_n)_{n\in \N}$, $\cH=(H_n)_{n\in \N}$ and $\cG=(G_n)_{n\in \N}$ such that $G'_n\leq H_n \leq G_n$ for all $n \in \N$,  $\DPSE{\mathcal{G'}}$ and $\DPSE{\mathcal{G}}$ are in \P but $\DPSE{\mathcal{H}}$ is \NP-hard.
\end{remark}

\section{Matrix groups in higher dimensions}
\label{se:high-dim-matrices}

In this section we are again in the \emph{matrix group} input model.
We first prove some \NP-hardnes result for the \emph{subgroup} input model for matrix groups. Then we will show that for generalized Heisenberg groups (\ref{Heisenberg}) spherical equations are solvable in polynomial time even if the dimension $n$ and the prime field $\Z_p$ are both part of the input. After that we will show 
$\DPSE{(\UT{4}{p})_{p \in \prm}} \in \P$.

\subsection{The subgroup input model}

\begin{corollary}
\begin{enumerate}[(a)]
    \item \label{SubgrGLn} For every constant $q \in \N \setminus\{0,  1\}$, $\DPSub{(\GL{n}{q})_{n \in \N}}$ and $\DPSub{(\SL{n}{q})_{n \in \N}}$ are $\NP$-complete. 
    \item \label{SubgrGLq} For every constant $n \in \N \setminus\{0,  1\}$, the problems $\DPSub{(\GL{n}{q})_{q \in \N}}$, $\DPSub{(\SL{n}{q})_{q \in \N}}$ and $\DPSub{(\TL{n}{q})_{q \in \N}}$ are $\NP$-complete. 
\end{enumerate}
\end{corollary}
In particular, $\DPSE{(\GL{2}{q})_{q \in \N}}$ and $\DPSE{(\TL{2}{q})_{q \in \N}}$ are in \P but the corresponding problem in the subgroup input model is \NP-hard. 
\begin{proof}
(a) The groups $S_n$ and $A_n$ are both generated by two elements (see e.g.\ \cite{Bray11}). We can embed $S_n$ into $\GL{n}{q}$ and $A_n$ into $\SL{n}{q}$. By \cref{thm: DPforSn}, \cref{cor: SEAn} and  \cref{prop:inNP},  part (\ref{SubgrGLn}) follows.

(b) Let $R, S \in \TL{n}{q}$, such that all entries coincide with the  identity matrix except for $R_{1,n}=1$ and $S_{n-1,n-1}=S_{n,n}=-1$.  Then, $R,S \in \SL{n}{q} $ and $\gen{R,S}\cong D_q$. Hence, part (\ref{SubgrGLq}) follows from \cref{thm: DnNPhard} and   \cref{prop:inNP}. 
\qed\end{proof}

\subsection{The Heisenberg groups}
The \emph{generalized Heisenberg group} $H^{(p)}_n$ of dimension $n$ over $\Z_p$ is the group
\small
\begin{align}\label{Heisenberg}
    H^{(p)}_n=\left\lbrace
\begin{pmatrix}
1 & \alpha_1& a_2\\
0& I & \alpha_3\\
0& 0& 1
\end{pmatrix}~\Big|~\alpha^T_1, \alpha_3 \in \Z_p^{n-2}, a_2\in \Z_p\right\rbrace
\end{align}
\normalsize
where $\alpha^T$ denotes the transposed vector/matrix of $\alpha$. 
 Observe that $H_3^{(p)}=\UT{3}{p}$ and $H^{(p)}_n \leq \UT{n}{p}$ for each $n$ and $p$. 
The following lemma gives necessary and sufficient conditions for a solution of $\DPSE{H^{(p)}_n }$ to exist.

\begin{lemma}\label{lem: FormulaUTthree} Let $X_i, C_i \in H^{(p)}_n$ for $i \in \interval{1}{k}$. Write $A=\prod_{i=1}^k X_iC_iX_i^{-1}$, 
\small
\begin{align*}
X_i=\begin{pmatrix}
1 & \chi_i& z_i\\
0& I & \gamma_i\\
0& 0& 1
\end{pmatrix},  ~C_i=\begin{pmatrix}
1 & \zeta_i^{(1)}& c_i\\
0& I & \zeta_i^{(3)}\\
0& 0& 1
\end{pmatrix},
~A=\begin{pmatrix}
1 & \alpha^{(1)}& a\\
0& I & \alpha^{(3)}\\
0& 0& 1
\end{pmatrix}
\end{align*}
\normalsize
as in (\ref{Heisenberg}). 
Then, $\alpha^{(1)}= \sum_{i=1}^k \zeta^{(1)}_{i}$, $
\alpha^{(3)}=\sum_{i=1}^k \zeta^{(3)}_i $ and
\begin{align*}
a=  \sum_{i=1}^k(\chi_i\cdot \zeta^{(3)}_i-\zeta^{(1)}_i\cdot \gamma_i) +  \sum_{i=1}^k \sum_{h=1}^{i-1} \zeta^{(1)}_h\cdot\zeta^{(3)}_i+\sum_{i=1}^k c_i 
\end{align*}
with $"+"$ being the componentwise addition for (row or column) vectors  and $\cdot$ the product of the vectors seen as matrices. 
\smallskip

In particular, $(C_1, \ldots, C_k)\in\DPSE{ H^{(p)}_n}$ if and only if $\alpha^{(1)}=\alpha^{(3)}=0$ and one of the following conditions holds: 
\begin{itemize}
    \item at least one $\zeta^{(1)}_{i}$, $\zeta^{(3)}_{i}$ is non-zero, or
    \item  $\sum_{i=1}^k c_i =0$.
\end{itemize}
\end{lemma}
\begin{proof}
We proof this by induction on $k$. For $k=1$ we have that

\small
\begin{align*}
X_1C_1X_1^{-1}=\begin{pmatrix}
1 & \zeta^{(1)}_1 &\chi_1\zeta^{(3)}_1 -\zeta^{(1)}_1\gamma_1+c_1\\
0 & 1 & \zeta^{(3)}_1 \\
0 & 0& 1
\end{pmatrix}.
\end{align*}
\normalsize
Assume that we already showed the claim for $k$. Then by induction we have for $a=  \sum_{i=1}^k(\chi_i\cdot \zeta^{(3)}_i-\zeta^{(1)}_i\cdot \gamma_i) +  \sum_{i=1}^k \sum_{h=1}^{i-1} \zeta^{(1)}_h\cdot\zeta^{(3)}_i+\sum_{i=1}^k c_i $ that 
\small
\begin{align*} 
\prod_{i=1}^{k+1} X_iC_iX_i^{-1}&=\begin{pmatrix}
1 & ~ \sum_{i=1}^k \zeta^{(1)}_i& a \\[2mm]
0 & I& \sum_{i=1}^k \zeta^{(3)}_i\\[2mm]
0 & 0& 1
\end{pmatrix} \cdot\begin{pmatrix}
1 & ~ \zeta^{(1)}_{k+1} &~ \chi_{k+1}\zeta^{(3)}_{k+1}-\zeta^{(1)}_{k+1}\gamma_{k+1}+c_{k+1}\\[2mm]
0 & I & \zeta^{(3)}_{k+1} \\[2mm]
0 & 0& 1
\end{pmatrix}\\[4mm]
&=\begin{pmatrix}
1 &~ \sum_{i=1}^{k+1} \zeta^{(i)}_1&  a' \\[2mm]
0 & I& \sum_{i=1}^{k+1} \zeta^{(i)}_3\\[2mm]
0 & 0& 1
\end{pmatrix}
\end{align*}
\normalsize
with 
\begin{align*}
a'&=\chi_{k+1}\zeta^{(3)}_{k+1}-\zeta^{(1)}_{k+1}\gamma_{k+1}+c_{k+1}+ \sum_{h=1}^k \left(\zeta^{(1)}_h \cdot\zeta^{(3)}_{k+1}\right)+a\\
 &\overset{I.H.}{=}\sum_{i=1}^{k+1}(\chi_i\zeta^{(3)}_i-\zeta^{(1)}_i\gamma_i)+  \sum_{i=1}^{k+1} \sum_{h=1}^{i-1} \zeta^{(h)}_1\cdot \zeta^{(3)}_i+\sum_{i=1}^{k+1} c_i. 
\end{align*}
As each component of $\chi_i$ and $\gamma_i$  appears only linearly in $a$, the equation $a=0$ can be solved under the conditions of the lemma.
\qed\end{proof}

We obtain the following theorem as an immediate consequence of \cref{lem: FormulaUTthree}:

\begin{theorem}\label{thm:heisenberg}
$\DPSE{(H_n^{(p)})_{(n,p) \in \N \times \prm}}$, $p$ given in binary, is in $\P$. 
\end{theorem}

Notice that $\mathcal{H}=(H_n^{(p)})_{(n,p) \in \N \times \prm}$  is a sequence of non-abelian matrix groups such that $\DPSE{\mathcal{H}} \in \P$ even if the dimension $n$ and the prime field $\Z_p$ are both part of the input.

\subsection{The sequence $(\UT{4}{p})_{p \in \prm}$}
In this section we show  $\DPSE{(\UT{4}{p})_{p \in \prm}} \in \P$. Observe that $H^{(p)}_4 \leq \UT{4}{p}$, so, in some sense, this is the largest sequence $\cG$ of non-abelian subgroups of $\GL{4}{p}$
for which we have shown $\DPSE{\cG}\in \P$.

 \begin{lemma}\label{prop: FormulaUTFour} Let $X_i, C_i \in \UT{4}{p}$ for $i \in \interval{1}{k}$. Write $D_i=X_iC_iX_i^{-1}$,  $A_i=\prod_{h=1}^iD_h$ and 
\small
\begin{align*}
X_i=\begin{pmatrix}
    1&x_i& w_i & u_i\\
    0&1&z_i &v_i \\
    0&0&1&y_i\\
    0&0&0&1
\end{pmatrix}, \hspace*{0.5cm} A_i=\begin{pmatrix}
1 & a^{(1)}_i & a^{(2)}_i & a^{(3)}_i\\[0.5mm]
0 & 1 & a^{(4)}_i & a^{(5)}_i \\[0.5mm]
0 & 0 & 1 & a^{(6)}_i\\[0.5mm]
0& 0& 0& 1
\end{pmatrix}. 
\end{align*}
The entries of $C_i$ and $D_i$ are denoted in the same way as the ones of $A_i$. Then, $d^{(j)}_i=c^{(j)}_i$ and $a^{(j)}_i=\sum_{h=1}^i c^{(j)}_h \text{ for } j \in \{1,4,6\}$ (in particular,  $a^{(j)}_0=0$), 
\begin{align*}
   d^{(2)}_i&= x_ic_i^{(4)} - c_i^{(1)}z_i + c_i^{(2)}, \qquad d^{(5)}_i= z_ic_i^{(6)} - c_i^{(4)}y_i + c_i^{(5)},\\
    d^{(3)}_i&= c_i^{(1)}y_iz_i - c_i^{(4)}x_iy_i + x_i c_i^{(5)} - c_i^{(1)}v_i + w_ic_i^{(6)}  - c_i^{(2)}y_i + c_i^{(3)},
\end{align*}
and 
\begin{align*}
   a^{(2)}_i&= \sum_{h=1}^i \left(d^{(2)}_h + a^{(1)}_{h-1} c^{(4)}_h\right), \qquad a^{(5)}_i= \sum_{h=1}^i \left( d^{(5)}_h + a^{(4)}_{h-1} c^{(6)}_h\right), \\
     a^{(3)}_i&= \sum_{h=1}^i \left(d^{(3)}_h+a^{(1)}_{h-1}d^{(5)}_h + a^{(2)}_{h-1}c^{(6)}_h \right).
\end{align*}
\normalsize
\end{lemma}
\begin{proof}
The entries of the matrices $D_i$ can be easily checked by multiplying the matrices by hand or using some computer algebra software. 
For $A_i$, we proceed by induction over $i$. For $i=1$ the claim follows because $a_0^{(j)}=0$ and $A_1=D_1$. 

\normalsize
Assume that we have already shown the formula for $A_i$. Then  for $A_{i+1}=A_i\cdot D_{i+1}$ we obtain 
\begin{align*}
&a^{(j)}_{i+1}=c_{i+1}^{(j)}+ \sum_{h=1}^{i} c_h^{(j)}=\sum_{h=1}^{i+1} c_h^{(j)} \qquad \text{ for } j \in \{1,4,6\},
\end{align*}
{\allowdisplaybreaks
\begin{align*}
a^{(2)}_{i+1}
&=d^{(2)}_{i+1}+a^{(1)}_id^{(4)}_{i+1}+\sum_{h=1}^{i}\left( d_h^{(2)}+ a^{(1)}_{h-1}c_h^{(4)}\right)
=\sum_{h=1}^{i+1} \left( d_h^{(2)}+ a^{(1)}_{h-1} c_h^{(4)} \right),\\
a^{(5)}_{i+1}
&=d^{(5)}_{i+1}+a_i^{(4)}d^{(6)}_{i+1}+\sum_{h=1}^i \left(d^{(5)}_h+a^{(4)}_{h-1}c^{(6)}_h\right) 
=\sum_{h=1}^{i+1} \left( d^{(5)}_h+a^{(4)}_{h-1}c^{(6)}_h \right),\\
a^{(3)}_{i+1}
&=d^{(3)}_{i+1}+a^{(1)}_id^{(5)}_{i+1}+a^{(2)}_id^{(6)}_{i+1}+\sum_{h=1}^i \left(d^{(3)}_h+a^{(1)}_{h-1}d^{(5)}_h+a^{(2)}_{h-1}c^{(6)}_h \right)\\
&=\sum_{h=1}^{i+1} \left( d^{(3)}_h+a^{(1)}_{h-1}d^{(5)}_h+a^{(2)}_{h-1}c^{(6)}_h \right).
\end{align*}
}
This shows the lemma. 
\qed\end{proof}

\begin{theorem}\label{thm:UTfour}
    $\DPSE{(\UT{4}{p})_{p \in \prm}}$ is in \P.
\end{theorem}

\begin{proof}
Let $X_i, C_i \in \UT{4}{p}$ for $i \in \interval{1}{k}$ as in \cref{prop: FormulaUTFour}.
We first check if  $\sum_{i=1}^k c_i^{(j)}=0$ for $j \in \{1,4,6\}$.
If this is not the case, then $(C_1, \ldots, C_k) \notin \DPSE{\UT{4}{p}}$ by \cref{prop: FormulaUTFour}.
Otherwise, we have to solve the following system of equations: 
\begin{align*}
    \sum_{i=1}^k \left(c^{(4)}_i x_i-c^{(1)}_iz_i+c^{(2)}_i +c_i^{(4)} a^{(1)}_{i-1} \right)&=0\\
       \sum_{i=1}^k \left(c^{(6)}_iz_i-c^{(4)}_iy_i+c^{(5)}_i+ c_i^{(6)} a^{(4)}_{i-1}\right)&=0\\
 \sum_{i=1}^k\Big( c^{(1)}_i(y_iz_i-v_i)-c^{(4)}_ix_iy_i+c^{(5)}_ix_i+c^{(6)}_iw_i-c^{(2)}_iy_i\qquad& \\{}+ c^{(3)}_i+ a^{(1)}_{i-1} \left(c^{(6)}_iz_i-c^{(4)}_iy_i+c^{(5)}_i\right)+c_i^{(6)}a^{(2)}_{i-1}\Big)&=0\qquad
\end{align*}
Note that $v_i,w_i,x_i,y_i,z_i$ are variables, $a^{(2)}_{i-1}$ also contains variables and  $a^{(1)}_{i-1}$ and $a^{(4)}_{i-1}$ are constants. 
For each $i$, $v_i$ and $w_i$ only appear linearly and only in the third equation. We have no restrictions on the variables, so if at least one  $c^{(1)}_i$ or $c^{(6)}_i$ and at least one $c^{(4)}_i$ is non-zero, we can solve the system.
If at least one $c^{(1)}_i$ or $c^{(6)}_i$ is non-zero but $c^{(4)}_i=0$ for all $i$ we need to check if the linear system consisting of the first two equations has a solution. This can be done in polynomial time (see e.g.\ \cite{KannanB79}). \\

Hence, from now on we assume that $c^{(1)}_i=c^{(6)}_i=0$ for all $i$.
We obtain the following system: 
\begin{align*}
&\sum_{i=1}^k c_i^{(4)}x_i+c^{(2)}_i=0\\
&\sum_{i=1}^k -c_i^{(4)}y_i+c^{(5)}_i=0\\
&\sum_{i=1}^k -c_i^{(4)}x_iy_i+c_i^{(5)}x_i-c_i^{(2)}y_i+c^{(3)}_i=0
\end{align*} 
We can assume that $c_i^{(4)}\neq 0$ for some $i$ because otherwise, the system is linear. W.l.o.g. $i=1$. 
Hence, the first two equations are equivalent to 

\begin{align*}x_1 &= \frac{-1}{c_1^{(4)}} \left(c_1^{(2)} + \sum_{i=2}^k \left( c_i^{(4)}x_i + c_i^{(2)}\right) \right)\qquad  \text{ and }\\ y_1 &= \frac{1}{c_1^{(4)}}\left(c_1^{(5)} + \sum_{i=2}^k \left( -c_i^{(4)} y_i+c_i^{(5)}\right)\right).
\end{align*}
By plugging this into the third equation, it remains to solve an equation of the form 
 
\begin{align} \label{SolvEqUTFour}
  \sum_{2\leq i, j \leq k} \alpha_{ij}x_iy_j+\sum_{i=2}^{k}( \beta_ix_i+\delta_iy_i)+\zeta=0
\end{align}
for suitable constants $\alpha_{ij},\beta_i$, $\delta_i$ and $\zeta$. Note that we do not get terms of the form $x_ix_j$. 
If $\alpha_{ij}=0$ for all $ i, j $, (\ref{SolvEqUTFour}) is a linear equation. Thus, we assume that there are $i_0,j_0$ with $\alpha_{i_0,j_0}\neq 0$. By setting $x_i=y_j=0$ for all $i \neq i_0$, $j \neq j_0$, we obtain the equation
\begin{align*}
\alpha_{i_0,j_0}x_{i_0}y_{j_0}+\beta_{i_0}x_{i_0}+\delta_{i_0}y_{i_0}+\zeta=0.
\end{align*}
Now, we choose $y_{j_0}$ such that $\alpha_{i_0,j_0}\cdot y_{j_0} +\beta_{i_0} \neq 0$. Then, 
\begin{align*}
    x_{i_0}=\frac{-\zeta-\delta_{i_0}y_{i_0}}{\alpha_{i_0,j_0}\cdot y_{j_0} +\beta_{i_0} }
\end{align*}
is a solution.
So also in this case we can solve the equation. This shows the proposition. 
\qed \end{proof}

\section{The groups $\Z_m^k\rtimes C_2$} \label{sec: grpZkC}
In this section we will show that there exist families of groups $\cG=(G_n)_{n \in \N}$ and 
$\cH=(H_n)_{n \in \N}$ such that $H_n \leq G_n$ for all $n$ and $\DPSE{\cG} \in \P$ (even in \ACC) but $\DPSE{\mathcal{H}}$ is \NP-hard. So we cannot expect that $\DPSE{\mathcal{G}} \in \P$ implies $\DPSE{\mathcal{H}} \in \P$.

As before let $C_2=\{\pm 1\}$. Then $C_2$ acts on $\Z_m$ by $(x,a) \mapsto xa$. The corresponding semi-direct product $\Z_m \rtimes C_2$ is precisely the dihedral group $D_m$.
We can define also the semi-direct product  $\Z_m^k\rtimes C_2=\{(a,x) ~|~a \in \Z^k_m, x=\pm 1\}$ with $C_2$ operating componentwise on $\Z_m^k$  meaning that the multiplication is given by 
\begin{align}\label{multplZkC}
    ((a_1, \ldots, a_k),x)\cdot ((b_1, \ldots, b_k),y)=((a_1 + x \cdot b_1, \ldots, a_k+ x \cdot b_k), x \cdot y).
\end{align}
The neutral element is $\smash{((\underbrace{0, \ldots, 0}_{k \text{ times }}),1)}$, which we also denote as $(0,1)$. We will show the following:

\begin{theorem} \label{prop: ZmC2}
Let $m = 3$ or $m \geq 5$ be fixed. Then $\DPSE{(\Z_m^{k} \rtimes C_2)_{k \in \N}}$ is \NP-complete.

\end{theorem}

Before we prove \cref{prop: ZmC2} we explain one of its consequences.
Consider the group $(\Z_m \rtimes C_2)^k$ with the operation of $C_2$ on $\Z_m$ defined as above. So the group multiplication is defined as follows:
\begin{align*} 
    ((a_1,x), \ldots, (a_k,x)) &\cdot ((b_1,y), \ldots, (b_k,y))\\ &\quad =((a_1+x \cdot b_1, x \cdot y), \ldots, (a_k+x \cdot b_k, x \cdot y)).
\end{align*}

We can embed the group $\Z^k_m \rtimes C_2$ into $(\Z_m \rtimes C_2)^k$ via \[((a_1, \ldots, a_k),x)\mapsto ((a_1,x), \ldots, (a_k,x)).\] 
Because $m$ is fixed,  $\Z_m \rtimes C_2$ and $\Z_m$ are fixed finite groups. Thus, according to \cref{lem: DirectProd}, $\DPSE{((\Z_m \rtimes C_2)^k)_{k \in \N}}$ and $\DPSE{(\Z_m^{k})_{k \in \N}}$ are in \P for $k$ given in unary. So we have the following corollary:

\begin{corollary} 
\label{prop: NPhardSubclass }
There exist sequences of groups $\mathcal{G'}=(G'_n)_{n\in \N}$, $\cH=(H_n)_{n\in \N}$ and $\cG=(G_n)_{n\in \N}$ such that $G'_n\leq H_n \leq G_n$ for all $n \in \N$,  $\DPSE{\mathcal{G'}}$ and $\DPSE{\mathcal{G}}$ are in \P but $\DPSE{\mathcal{H}}$ is \NP-complete. 
\end{corollary}

\begin{proof}[of \cref{prop: ZmC2}]
Membership in \NP is by \cref{prop:inNP} (note that $\Z_m \rtimes C_2$ embeds into 2-by-2 matrices as described in \cref{sec: Dihrgrps}; hence, $\Z_m^k \rtimes C_2 \hookrightarrow (\Z_m^{k})_{k \in \N}$ can be embedded into $2m$-by-$2m$ matrices).

To show \NP-hardness, we reduce from \esc (see \cref{sec:NP}). 
Let $A_1, \ldots, A_\ell $ be subsets of $\interval{1}{k}$. We consider the group \[H=\Z^{k+\ell}_m\rtimes C_2.\]
Write $a_{i,j}$ for the $j$-th entry of $a_i \in  \Z^{k+\ell}_m$. We define $2\ell$ elements $c_i=(a_i,1) \in H$.  For $i \in \interval{1}{\ell}$, $j \in \interval{1}{k}$ we set $a_{i,j}=a_{\ell+i,j}=1$ if $j \in A_i$, $a_{\ell+i,k+i}=1$ and $a_{i,j}=0$ otherwise. 
We show that the equation 
\begin{align} \label{SEWP}
    \prod_{i = 1}^{2\ell} z_ic_iz_i^{-1}=((\underbrace{2, \ldots, 2}_{k \text{ times } },\underbrace{1, \ldots, 1}_{\ell \text{ times }}),1)
\end{align}
has a solution if and only if $\interval{1}{k}$ together with $A_1, \ldots, A_\ell$ is a positive instance of  \esc. 
Let  $\beta=(0,-1)$. We will use the observation that 
\begin{align} \label{cancel}
    \beta c_{i} \beta^{-1}c_{\ell+i} = ((\underbrace{0, \dots, 0}_{\mathclap{k + i-1 \text{ times}}},1,0,\dots,0),1).
\end{align}
Let $I \subseteq \interval{1}{\ell} $ such that $\{A_{i}\mid i\in I\}$ is an exact covering. By the choice of the $c_i$ we obtain that \[ \prod_{i \in I} c_{i}c_{\ell+i} = ((\underbrace{2, \ldots, 2}_{k \text{ times }},b_1, \ldots, b_\ell),1),\] 
with $b_i=1$ for $i \in I$  and $b_i=0$ otherwise. 
Now, (\ref{cancel}) implies  \[\prod_{i \in I} c_{i} \cdot c_{\ell+{i}} \cdot \prod_{i \not\in I}\beta c_{i} \beta^{-1} \cdot c_{\ell+i} = ((\underbrace{2, \ldots, 2}_{k \text{ times } },\underbrace{1, \ldots, 1}_{\ell \text{ times }}),1).\] 
 Thus, according to \cref{lem:reorder} there exists a solution to (\ref{SEWP}).

Now assume that (\ref{SEWP}) has a solution. Because $(c,1)^{(a,x)}=(c,1)^{(0,x)}$ there is a solution such that $z_i=(0,\pm 1)$ for all $i$. We have $z_{\ell+j}=(0,1)$ for all $j\in \interval{1}{\ell}$, as otherwise we get a $-1$ in one of the last $\ell$ components. 
Let $I=\Set{i \in \interval{1}{\ell}}{z_i=(0,1)}$.  Again by (\ref{cancel}),  it follows that $\sum_{i \in I}(a_{i,j}+a_{\ell+i,j})= 2$ for all $j\in \interval{1}{k}$ (be aware that this is in $\Z_m$). Now, it remains to combine the following observations 
\begin{itemize}
    \item $a_{i,j}+a_{\ell+i,j}\in \{0,2\}$ for all $i\in I$ and  $j \in \interval{1}{k}$,
    \item each $j$ appears in at most three sets $A_i$ meaning that  for each $j$ we have $\abs{\{ i \in \interval{1}{\ell} \mid a_{i,j}+a_{\ell+i,j} = 2 \}}  \leq 3$), and
    \item  $m=3 $ or $m\geq 5$
\end{itemize} 
to conclude that for each $j$ there is exactly one $i \in I$ with  $a_{i,j}=1$. Thus the $A_{i}$ for $i \in I$ are  an exact covering of $\interval{1}{k}$. 
\qed\end{proof}

\section{Conclusion and open questions}

As one of our main results, we have shown that $\DPSE{(\GL{2}{p})_{p \in \prm}}$ is in \P.  Our result about $(\UT{4}{p})_{p \in \prm}$ shows that also in dimension bigger than two there are sequences of non-abelian matrix groups such that $\DPSE{\cG} \in \P$. With the the generalized Heisenberg groups we have found a sequence of matrix groups for which we could show that  $\DPSE{\cG} \in \P$  even if the dimension is part of the input too. 

The dihedral groups as well as the groups $S_n$ and $A_n$ also can be seen as matrix groups. So, with this embedding, we have sequences of matrix groups  with \NP-hard Diophantine problem for spherical equations. Indeed, \cref{sec: grpZkC} shows that \NP-hardness does not necessarily transfer to super- nor sub-groups. So we are left with the following open questions:

\begin{itemize}
\item What is the complexity of $\DPSE{(\GL{n}{p})_{p \in \prm}}$ for $n \geq 3$?
We conjecture that for every $n$ there is a number $f_n$ such that every spherical equation over $\GL{n}{p}$ consisting of at least $f_n$ non-scalar matrices has a solution~-- so for each fixed $n$ the problem would be still solvable in polynomial time. However, this might change if the dimension $n$ is part of the input.

\item  What about other sequences of matrix groups in higher dimensions, like for example for $(\UT{5}{p})_{p \in \prm}$?

\item What is the role of prime vs.\ arbitrary integer as modulus? In particular, is $\DPSE{(D_p)_{p \in \prm}}$ \NP-hard? An affirmative answer to this question seems very unlikely without solving deep number-theoretic questions (see \cref{rem: CramerforET}). 

On the other hand, is it possible to extend the polynomial-time algorithm for $\DPSE{(\ET{2}{p})_{p \in \prm}}$ to $\DPSE{(\ET{2}{n})_{n \in \N}}$? A possible obstacle here might be that one might have to factor $n$ here, but it also seems plausible that factoring can be circumvented for the decision variant (as opposed to the variant where we want to compute a solution). The same thoughts apply to $(\TL{2}{n})_{n \in \N}$.

    \item What makes spherical equations difficult? Clearly the groups must be non-abelian~-- but on the other hand the example of $\GL{2}{p}$ shows that ``too far from abelian'' might be also easy.
\end{itemize}

\newcommand{\Ju}{Ju}\newcommand{\Ph}{Ph}\newcommand{\Th}{Th}\newcommand{\Ch}{Ch}\newcommand{\Yu}{Yu}\newcommand{\Zh}{Zh}\newcommand{\St}{St}\newcommand{\curlybraces}[1]{\{#1\}}

\end{document}